%% file: friends.tex
\documentclass{amsart} 

\usepackage{amsmath,amssymb,amsthm} 
\usepackage{graphicx} 
\usepackage[colorinlistoftodos]{todonotes}
\usepackage{subfig}
\usepackage[arrow, matrix, curve]{xy} 
\usepackage{xcolor} 
\usepackage{makecell} 
\usepackage{multirow} 
\usepackage{enumerate}
\usepackage{booktabs} 
\usepackage{siunitx} 
\usepackage{hyperref} 
\usepackage{algorithm}
\usepackage{longtable}
\usepackage{cleveref}

\newcommand{\Z}{\mathbb{Z}}

\newtheoremstyle{thm}{}{}{\itshape}{}{\bfseries}{}{ }{} 
\newtheoremstyle{definition}{}{}{}{}{\bfseries}{}{ }{} 

\theoremstyle{thm}
\newtheorem{Theorem}{Theorem}[section]
\newtheorem{thm}[Theorem]{Theorem}

\newtheorem{prop}[Theorem]{Proposition}
\newtheorem{cor}[Theorem]{Corollary}
\newtheorem*{Theorem-ohne}{Theorem}
\newtheorem{con}[Theorem]{Conjecture}

\newtheorem{ques}[Theorem]{Question}

\newtheorem*{alg:friend_search}{Algorithm \ref{alg:friend_search}}
\newtheorem*{thm:census_friends}{Theorem \ref{thm:census_friends}}
\newtheorem*{prop:algo2}{Algorithm \ref{prop:algo2}}
\newtheorem*{conj:WHD}{Conjecture \ref{conj:WHD}}

\theoremstyle{definition}
\newtheorem{defi}[Theorem]{Definition}
\newtheorem{rem}[Theorem]{Remark}
\newtheorem{ex}[Theorem]{Example}

\crefname{con}{Conjecture}{Conjectures}
\crefname{ques}{Question}{Questions}
\crefname{cor}{Corollary}{Corollaries}
\crefname{prob}{Problem}{Problems}
\crefname{obs}{Observation}{Observations}
\crefname{alg}{Algorithm}{Algorithms}
\crefname{strat}{Strategy}{Strategies}
\crefname{lem}{Lemma}{Lemmas}
\crefname{prop}{Proposition}{Propositions}

\begingroup\expandafter\expandafter\expandafter\endgroup
\expandafter\ifx\csname pdfsuppresswarningpagegroup\endcsname\relax
\else
\fi

\numberwithin{equation}{section}

\title{The search for exotic knot traces}
\author{Marc Kegel}
\address{Universidad de Sevilla, Dpto.\ de Álgebra,
Avda.\ Reina Mercedes s/n,
41012 Sevilla}
\email{mkegel@us.es, kegelmarc87@gmail.com}
\author{Jonathan Spreer}
\address{School of Mathematics and Statistics, The University of Sydney, Carslaw Building F07, Camperdown NSW 2006, Australia}
\email{jonathan.spreer@sydney.edu.au}
\date{\today}

\begin{document}

\begin{abstract}
Two distinct knots are said to be {\em friends} if their complements, filled along the $0$-slope, produce diffeomorphic $3$-manifolds. In this article, we develop a practical algorithm, implemented using {\em SnapPy} and {\em Regina}, to search for a friend of a given knot. As an application, we construct a census of simple knots that admit friends and use these data to formulate conjectures about knot friends.
\end{abstract}


\makeatletter
\@namedef{subjclassname@2020}{%
  \textup{2020} Mathematics Subject Classification}
\makeatother



\maketitle

\section{Introduction}

We say that two knots in the $3$-sphere $S^3$ are \textit{friends} if they have diffeomorphic $0$-surgeries. 
In the literature, there exist many examples of pairs of knots that are friends, see for example~\cite{Osoinach_annulus,Miller_Piccirillo_traces,Piccirillo_shake_genus,Manolescu_Piccirillo_0_surgery}. On the other hand, many simple knots (like the unknot, the trefoil, or the figure eight) are known to have no friends~\cite{Gabai,Ozvath_Szabo_trefoil_figeight,Baldwin_Sivek_0_inf}. 
The study of friends is motivated by the observation that a pair of friends with different smooth sliceness status\footnote{Recall that a knot is \textit{smoothly slice} if it bounds a smoothly embedded disk in the $4$-ball.} yields a counterexample to the smooth $4$-dimensional Poincaré conjecture~\cite[Problem 1.19]{Kirby1997Problems}, cf.~\cite{Freedman_Gompf_Man_Machine,Manolescu_Piccirillo_0_surgery}.

 Apart from the possibility of using friends to disprove the smooth $4$-dimensional Poincaré conjecture, this article is motivated by the following questions.

\begin{ques} \label{ques:main}
  Given a knot $K$, does $K$ have a friend?
\end{ques}

\Cref{ques:main} should be understood to describe general topological, algebraic, or geometric conditions on $K$ that ensure the existence or non-existence of a friend.

\begin{ques}\label{ques:decidable}
  Is there an algorithm that, given a knot $K$, returns \texttt{true} if $K$ has a friend, and \texttt{false} otherwise? 
\end{ques}

To make \Cref{ques:decidable} more precise we can assume that the input knot $K$ is given as a knot diagram, or as a triangulation of the exterior of $K$. In the former case, we often abuse notation and refer to $K$ as both the knot and its representation as an input diagram.

\subsection{Algorithmic search for friends}

The main contribution of this article is code and the results of running this code. Namely, we have the following result.

\begin{alg:friend_search}[Friend search]
There exists a practical and implemented algorithm that takes as input a diagram of a knot $K$, or a triangulation of the exterior of $K$, and outputs a (possibly empty) list of friends of $K$.
\end{alg:friend_search}

Note that the solution of the homeomorphism problem for closed $3$-manifolds or $3$-manifolds with toroidal boundary~\cite{Kuperberg} together with the Gordon--Luecke theorem~\cite{Gordon_Luecke} implies the existence of an algorithm that takes as input a knot $K$ and outputs in finite time any given friend of $K$ -- if it exists: Enumerate all knots $K'$ by increasing crossing number, create their $0$-surgeries $K'(0)$, and check if $K(0)$ and $K'(0)$ are diffeomorphic. This approach was taken in~\cite{Kegel_Weiss} for knots of low complexity. However, since the enumeration of knots has super-exponential running time in the crossing number, this is impractical in general. Nonetheless, on a theoretical level, \Cref{ques:decidable} reduces to certifying that a given knot has no friend.

\Cref{alg:friend_search}, on the other hand, takes a knot diagram of $K$ and first constructs a $1$-vertex triangulation of its $0$-surgery. 
It then removes an embedded closed curve $c$ from $K(0)$, and checks if $K(0)\setminus c$ is the complement of a knot in $S^3$ using the $6$-theorem~\cite{Ag00,La00}. This procedure is then iterated over all curves $c$ in $K(0)$ up to ambient isotopy (i.e.\ over all knots in $K(0)$) using a sensibly chosen order.

For producing a $1$-vertex triangulation of $K(0)$ from the input diagram, we use readily available functionality of {\em SnapPy}~\cite{SnapPy}. For enumerating a list of curves $c$ to remove from $K(0)$ we use the fact that every edge in a $1$-vertex triangulation describes an embedded closed curve in the underlying manifold: We first produce a list of $1$-vertex triangulations of $K(0)$ by exhaustively modifying the initial triangulation using local moves (e.g.\ bi-stellar flips \cite{Pachner,Rubinstein2019}). Then every edge of every modified triangulation is added to the collection of curves. Both, exhaustive enumeration of ``neighbor triangulations'', as well as drilling out edges from a triangulation is readily available functionality of {\em Regina}~\cite{Regina}. We describe this algorithm in \Cref{sec:friend_search}.

By standard arguments (carried out rigorously in a more specialised setting in \cite{knotted_edge}), we can expect every ambient isotopy class of a closed curve in $K(0)$ to be realised as an edge in some triangulation of $K(0)$. It follows that this approach will, eventually, find every given friend $K'$ of $K$. Worst case running times of this procedure quickly become excessive -- even if $K'$ is observed by an edge in a triangulation of $K(0)$ that is only a small number of local moves away from the initial triangulation of $K(0)$. In practice, however, the approach is efficient in many cases, as we discuss in \Cref{sec:friend_search} and \ref{sec:census} below.

A similar algorithm for hyperbolic $0$-surgeries is described in~\cite{Dunfield_exterior_to_link} and further extended and used in~\cite{Dunfield_Gong}. There, the curves $c$ are enumerated by considering paths in the $1$-skeleton of a dual cell-decomposition, or by drilling out short geodesics. 

\subsection{A census of friends}

We apply the friend search -- \Cref{alg:friend_search} -- to standard datasets of knots: The \textit{low crossing number knots} (i.e.\ prime knots that admit diagrams with at most $12$ crossings) and the \textit{SnapPy census knots} (i.e.\ the hyperbolic knots whose complements can be ideally triangulated by at most $9$ ideal tetrahedra). Low-crossing number knots are described using the Dowker–Thistlethwaite (DT) notation~\cite{DT_notation}, as in the Hoste–Thistlethwaite–Weeks table~\cite{HTW_table}, cf.~\cite{KnotInfo}. Census knots are described using the notation of Dunfield~\cite{Dunfield_census} as implemented in {\em SnapPy}~\cite{SnapPy}. For example, the figure eight knot is $K4a1$ in the DT notation and $K2\_1$ in the census knot notation. For more background on the {\em SnapPy} census knots, we refer to~\cite{Dunfield_L_space,ABG+19,Baker_Kegel_braid,BKM_alt,BKM_QA,census10} and references therein.

\begin{thm:census_friends}\hfill
\begin{enumerate}
    \item Among the $2977$ low crossing number knots, at least $1867$ have a friend.
    \item Among the $1267$ {\em SnapPy} census knots, at least $ 183$ have a friend.
\end{enumerate}
\end{thm:census_friends}

The low crossing number knots that admit a friend are displayed in \Cref{tab:low_crossing}. PD codes of potentially minimal diagrams of the friends can be accessed at~\cite{data}. Their average crossing number is $159$, while the median is $109$. The friend with a diagram of maximal crossing number is a friend of $K12a951$ and has $1259$ crossings. For the {\em SnapPy} census knots, the average crossing number is $53$, the median is $42$, and the maximal crossing number diagram of a friend has $321$ crossings. The census knots with a friend are listed in \Cref{tab:census}. In particular, this recovers all friends found in~\cite{Kegel_Weiss}.

By analyzing our data, we observe that no fibered and strongly quasi-positive knot shows up in our tables. Among the low-crossing number knots, there are exactly $42$ (i.e.\ around 1\%) fibered and strongly quasi-positive (or strongly quasi-negative) knots~\cite{KnotInfo}, while  exactly $739$ of the {\em SnapPy} census knots (i.e.\ around 60\%) are strongly quasi-positive (or strongly quasi-negative) and fibered~\cite{CensusKnotInvariants}. This motivates the following conjecture. 

\begin{con}\label{conj:fibSQP}
    Fibered and strongly quasi-positive (or strongly quasi-negative) knots do not have friends.
\end{con}

Note that \Cref{conj:fibSQP} implies an affirmative answer to Question 1.7 from~\cite{Baldwin_sivek_traces}.

\subsection{Piccirillo friends}

Let $K$ be a knot with unknotting number one given by a diagram $D$ with an {\em unknotting crossing} $x$, i.e.\ a crossing such that $D$ after changing the crossing at $x$ represents the unknot. In \cite{Piccirillo_Conway_knot}, Piccirillo describes a simple, elegant method for constructing a knot $K_x$, the \emph{Piccirillo friend} of $K$, whose trace is diffeomorphic to the trace of $K$. If $D$ and $x$ are given, the construction is algorithmic, see \Cref{alg:Piccirillo}. An implementation of the algorithm was provided in \cite{Kegel_Weiss}. Using this code, we construct the Piccirillo friends for all low crossing number knots and all {\em SnapPy} census knots with unknotting number one.

As the notation suggest the Piccirillo friend construction is independent of the diagram $D$. Instead, it needs as input an {\em unknotting arc} $x$ (i.e.\ an arc connecting two points on the knot such that, if projected onto each other in a diagram, the corresponding crossing becomes an unknotting crossing). 

Piccirillo friends of twisted Whitehead doubles are isotopic to the original knot; see for example \Cref{prop:Wdouble} below. Regarding the converse, we verify that for all but the $15$ low crossing number knots listed in Table~\ref{tab:twist}, the knot $K_x$ is not isotopic to $K$. These $15$ knots are readily checked to be twist knots, and hence twisted Whitehead doubles of the unknot, see \Cref{prop:dual_census} below. This motivates the following conjecture.

\begin{conj:WHD}
The Piccirillo friend $K_x$ of a knot $K$ is isotopic to $K$ if and only if $K$ is a twisted Whitehead double.
\end{conj:WHD}

\begin{rem}
    \Cref{conj:WHD} does not contradict \Cref{conj:fibSQP}. Indeed, for a strongly quasi-positive knot, the $3$-genus agrees with its smooth $4$-genus~\cite{Rudolph}, which is a lower bound on the unknotting number. Thus, a strongly quasi-positive knot with unknotting number $1$ also has $3$-genus $1$. The only fibered knots with $3$-genus $1$ are the trefoil and the figure eight knot. Both have no friends by~\cite{Gabai}. 
    
    We also remark that \Cref{conj:fibSQP} is about isotopy. There exist knots $K$ that are not isotopic to their mirrors $\overline{K}$ such that $K$ and $\overline{K}$ are friends. These can also arise as Piccirillo friends. The simplest such knot is $K13n469$. We refer to \cite{Kegel_Weiss} for details.
\end{rem}

\begin{table}[htbp] 
\caption{The twist knots among the low crossing number knots and the {\em SnapPy} census knots.}
\label{tab:twist}
\begin{tabular}{@{}llllllll@{}}
\toprule
  $K3a1$ & $K4a1$ & $K5a1$ & $K6a3$ & $K7a4$ & $K8a11$ & $K9a27$ & $K10a75$ \\ 
  $K11a247$ & $K12a803$ & $K7\_2$ & $K8\_1$ & $K8\_2$ & $K9\_1$& $K9\_2$ & \\
\bottomrule
\end{tabular}
\end{table}

\subsection{Concordance friends}

In the last part of the article, we present another (so far unsuccessful) method to search for friends with different sliceness status. We say that two knots $K$ and $K'$ are \textit{concordance friends} if there exists a sequence of knots 
$$K=K_1,K_2,K_3,\ldots,K_{n-1},K_n=K'$$ 
such that for every $i$, the knots $K_i$ and $K_{i+1}$ are either concordant or friends. Since concordant knots have the same smooth $4$-genus and thus the same sliceness status, we immediately obtain the following.

\begin{cor}
    If there exist concordance friends with different sliceness statuses, then there exists a smooth $4$-manifold that is homotopy equivalent to $S^4$ but not diffeomorphic to $S^4$.\qed
\end{cor}

To search for concordance friends, we attach ribbon bands (see \Cref{def:Ribbon_band}) to a given knot $K$ to create knots that are concordant to $K$.

\begin{prop:algo2}[Concordance search]
    There exists a practical and implemented algorithm that takes as input a knot diagram of a knot $K$ and outputs arbitrarily many knots $K'$ $($presented as diagrams$)$ that are concordant to $K$.
\end{prop:algo2}

In Section~\ref{sec:concordance_friends}, we explain \Cref{prop:algo2} in more detail and combine it with \Cref{alg:friend_search} to obtain an algorithm that searches for concordance friends of an input knot, see \Cref{algo:concordance_friend_search}. As an application, we explain a new strategy to search for concordance friends with different sliceness status, see \Cref{algo:conc_friends}. We also apply our code to find concordance friends of many torus knots, cables, and Whitehead doubles. Note that it follows from~\cite[Theorem 1.2]{BKM_traces} that any knot has a concordance friend (that is not obviously concordant to $K$).

\subsection*{Code and data}
The code and additional data are available at~\cite{data}.

\subsection*{Acknowledgments}
We thank Nathan Dunfield, Shelly Harvey, Lisa Piccirillo, and Arunima Ray for useful discussions. We use code and data from~\cite{SnapPy,Regina,sagemath,Dunfield_census,Dunfield_exterior_to_link,FPS_code,BKM_alt,BKM_QA,Kegel_Weiss,KnotInfo,CensusKnotInvariants,KLO,KnotJob,Dunfield_Gong}. Figures \ref{fig:K6a2} and \ref{fig:Conway} were created with KLO~\cite{KLO} and Figures \ref{fig_Wp_friend} and \ref{fig_Wn_friend} with {\em SnapPy}~\cite{SnapPy}.

\subsection*{Individual grant support}
MK is supported by the DFG, German Research Foundation, (Project: 561898308); by a Ram\'on y Cajal grant (RYC2023-043251-I) and the project PID2024-157173NB-I00 funded by MCIN/AEI/10.13039/50110001\-1033, by ESF+, and by FEDER, EU; and by a VII Plan Propio de Investigación y Transferencia (SOL2025-36103) of the University of Sevilla. JS is supported by the Australian Research Council under the Discovery Grant scheme, grant number DP220102588.

\section{The friend search}\label{sec:friend_search}

In this section, we present \Cref{alg:friend_search} to search for a friend of a given knot $K$. 

\begin{algorithm}
\caption{Algorithm to search for friends of a given knot $K$.}
\label{alg:friend_search}
\begin{description}
    \item[Input] A knot diagram $D$ representing the knot $K$.
    \item[Output] A (possibly empty) list of knot diagrams representing friends of $K$.
\end{description}
\begin{enumerate}
    \item In {\em SnapPy}, construct an ideal triangulation of the knot complement of $K$.
    \item Create the $0$-surgery $K(0)$ on this ideal triangulation.
    \item Send $K(0)$ to {\em Regina} to obtain a $1$-vertex triangulation $T$ of $K(0)$. 
    \item For each edge $e$ in $T$:
    \begin{enumerate}
        \item Drill out $e$ to obtain an ideal triangulation of $E=K(0)\setminus e$.
        \item Check if $E$ is diffeomorphic to the complement of a knot $K'$ in $S^3$:
            \begin{enumerate}
                \item Check if $H_1(E)$ is isomorphic to $\Z$. If it is not, $E$ is not a knot complement, skip to the next edge $e$.
                \item If $H_1(E)$ is isomorphic to $\Z$, look up $E$ as a triangulation of a known knot complement. If it is, add $E$ to a list of potential friends, and skip to the next edge.
                \item If not, send $E$ to {\em SnapPy} and search for a hyperbolic structure.
                \item If {\em SnapPy} finds a hyperbolic structure, compute all slopes of length less than $6$. 
                \item For each such slope $s$:
                \begin{enumerate}
                    \item Create the filled manifold $E(s)$.
                    \item Check if $E(s)$ is diffeomorphic to $S^3$ (by either simplifying its fundamental group, or running $3$-sphere recognition in {\em Regina}).
                    \item If it is, add $E$ to the list of potential friends, skip to next edge $e$.
                \end{enumerate}
            \end{enumerate}
    \end{enumerate}
    \item Using a mix of knot invariants and comparisons, sort out duplicates in the list of potential friends to obtain a list of knot exteriors representing certifiably distinct knots different to $K$.
    \item If the list of friends is empty (and some upper bound on running time is not met), apply Pachner moves to $T$, and go back to (4).
    \item If friends have been found (represented by ideal triangulations), create their diagrams using the method described in~\cite{Dunfield_exterior_to_link}.
\end{enumerate}
\end{algorithm}

\subsection{Correctness and caveats of \Cref{alg:friend_search}}

In this section, we discuss why \Cref{alg:friend_search}, modulo some caveats and adjustments, will find any given friend $K'$ of the input knot $K$ in finite time.

For Steps (1)-(3), note that {\em SnapPy} uses an algorithm originally due to Weeks \cite{WeeksTriangulation} to construct an ideal triangulation of the knot complement of $K$ from a diagram of $K$. The algorithm is efficient and guaranteed to return an ideal triangulation with number of tetrahedra at most four times the number of crossings of the input diagram, plus a small additive constant. Permanently filling the torus cup of the resulting ideal triangulation and sending it to Regina to obtain a $1$-vertex triangulation $T$ of the closed $3$-manifold $K(0)$ is again deterministic and efficient.

For Step (4), note that in a $1$-vertex triangulation of a closed $3$-manifold, every edge is a loop edge, and hence an embedding of the circle into that manifold. In particular, every edge $e$ of $T$ represents a knot in $K(0)$. Moreover, every (PL) embedding of a circle into a closed $3$-manifold can be triangulated with the circle being represented as a sequence of edges. Using standard moves and simplification procedures, this circle can always be transformed into a loop, and the surrounding triangulated $3$-manifold can always be turned into a $1$-vertex triangulation. The latter statement technically has a very weak and detectable caveat that does not play a role in our situation, see \cite{burton12-unknot} for details. It follows that, if $K$ has a friend $K'$, then there exists a triangulation $T'$ of $K(0)$ with an edge representing $K'$. Moreover, as constructions from~\cite{knotted_edge,knot_factorisation} show, such a triangulation need not necessarily be excessively large. Hence, we can hope for $K'$ to be represented by an edge in a triangulation $T'$, where $T'$ is either $T$ itself or a slight modification thereof. In fact, this hope is confirmed by the results of our experiments as presented in \Cref{tab:low_crossing}. 

Step (4)(b) is implemented as a heuristic. Computing the homology of a manifold is a straightforward, polynomial time procedure. Looking up $E$ in a list of known knot complements is as well. This can fail, however, if $E$ does not simplify to a knot complement triangulation in the database. In rare instances, we may even miss a case where $E$ is the complement of input knot $K$. However, theoretically this can be avoided since the homeomorphism problem for $3$-manifolds is solved~\cite{Kuperberg}, and practically, this can be avoided by computing enough knot invariants -- revealing that this problem almost never appears. Similarly, {\em SnapPy} may not find a hyperbolic structure on $E$, despite $E$ being hyperbolic. Experiments with Dunfield's {\em Regina} recognition code~\cite{Dunfield_census}, cf.~\cite{FPS_code}, show that this almost never happens. Theoretically, it is decidable if a manifold is hyperbolic~\cite{decide_hyperbolic}, and thus the algorithm can be made rigorous here as well. Moreover, if {\em SnapPy} does not find a hyperbolic structure on $E$, we can also use Dunfield's {\em Regina} recognition code to check if $E$ is the complement of a non-hyperbolic knot in $S^3$. 

For Step (4)(b)(iv), it follows from the 6-theorem~\cite{Ag00,La00} that, if $E$ is a knot complement, then there exists a filling along a slope of length less than $6$ yielding $S^3$. Moreover, since $3$-sphere recognition is solved and implemented in {\em Regina}, if a filling yields $S^3$, we are guaranteed to find it (although, in practice, we do not run {\em Regina}'s $3$-sphere recognition routine, but a heuristic to simplify the fundamental group of the filled manifold, using the resolution of the Poincar\'e conjecture~\cite{perelman1,perelman2}). 

Finally, for every triangulation of a knot complement $E$ in $S^3$ found by our algorithm, by construction, the corresponding knot $K'$ admits a Dehn filling to $K(0)$. Since a $p/q$-surgery on a knot $K$ has first homology isomorphic to $\Z_p$, it follows that the $0$-surgery of $K'$ is diffeomorphic to $K(0)$ and thus $K$ and $K'$ are friends. 

In Step (6), since any two $1$-vertex triangulations of the same closed 3-manifold are connected by a finite sequence of Pachner moves \cite{matveev03-algms,Pachner}, and assuming that we never miss identifying a knot complement in $S^3$, any exhaustive breadth first search through the Pachner graph will find every friend of $K$ in finite time. In practice, the {\em Regina} function \texttt{retriangulate} performs exactly that up to a specified depth, and we use this function for our friend search. However, this explores many very similar triangulations with edges representing very similar curves. To counterbalance this, we also run a variant where the next triangulation is determined by the Markov chain style random walk from \cite{Altmann02012026}. In both methods, we keep a list of triangulations of $K(0)$, and candidate knot complements $E$ to avoid testing the same triangulations and curves multiple times.

For Step (7), given a triangulation of the complement $E$ of a friend $K'$, we use the method from~\cite{Dunfield_exterior_to_link} to turn $E$ into a diagram of $K'$. For this, we save the $S^3$-filling of $E$, change the basis on the cusp of $E$ such that $1/0$ corresponds to the meridian, and then run the method to create a diagram of $K'$. In a second step, we use {\em SnapPy} to simplify the initial diagram and finally record its PD code. Note that this method is a heuristic and is not always guaranteed to work. But again theoretically there is an algorithm creating a knot diagram as a corollary of the resolution of the homeomorphism problem for $3$-manifolds with toroidal boundary. 

\subsection{Other methods to find friends}
\label{rem:options}

If the $0$-surgery $K(0)$ of $K$ is hyperbolic, there are two other options to drill curves from $K(0)$ as done in~\cite{Dunfield_exterior_to_link} and~\cite{Dunfield_Gong}. There, curves $c$ are enumerated by considering paths in the $1$-skeleton of a dual cell-decomposition, or by drilling out short geodesics. Both options work in {\em SnapPy} without the use of {\em Regina}. 
\begin{enumerate}
    \item For the first option we take the triangulation $T$ of $K(0)$ and create closed loops $c$ in the $1$-skeleton of the dual cell-decomposition of $T$. Using {\em SnapPy}, we drill out these curves $c$ and proceed as in the above algorithm. By modifying $T$, we obtain different dual curves. However, since the way $T$ is modified -- as implemented in {\em SnapPy} -- is not exhaustive, this search sometimes fails to find existing friends. 
    \item For the second option, we compute the length spectrum of $K(0)$ up to a certain bound of the length. Then we drill out the geodesics in the length spectrum and proceed as above. Note that this algorithm is at the moment not guaranteed to find any given friend, since it is not known if the dual curve of a friend $K'$ is always a geodesic. For further discussion, we refer to~\cite{Dunfield_Gong}.
\end{enumerate}
We implemented the above two methods. They find all low-complexity friends that were found by \Cref{alg:friend_search}. However, there are significant differences in running times. In some examples, \Cref{alg:friend_search} is much quicker to find friends, while in other examples, the methods presented here are quicker. Thus, in practice, running all three methods in parallel may be optimal. We refer to \Cref{sec:runtime} for further comparison of the runtimes of the different methods.

\subsection{Notable results of running \Cref{alg:friend_search}}

We present a few examples of friends that we found by running \Cref{alg:friend_search}.

\begin{ex}
The simplest knot for which we found a friend is the knot $K=K6a2$ shown on the left of \Cref{fig:K6a2}. \Cref{alg:friend_search} shows that $K$ has the $20$-crossing friend $K'$ shown on the right of \Cref{fig:K6a2}. From that diagram, it is straightforward to see that $K'$ is only one crossing change away from a ribbon knot and thus the smooth $4$-genus of $K'$ is at most $1$. Since $K$ has topological $4$-genus $1$ it follows that $K'$ also has topological and thus smooth $4$-genus $1$~\cite{Freedman}. 

Note that $K$ has unknotting number $1$, and $K'$ can also be recovered as a Piccirillo friend, see Section~\ref{sec:Piccirillo}.
\end{ex}

\begin{figure}[htbp]
     \centering
     \includegraphics[width=.8\textwidth]{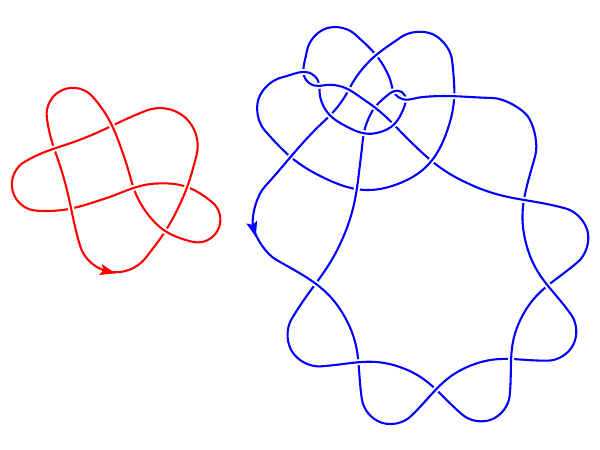}  
     \caption{Left: A diagram of $K6a2$. Right: A diagram of a $20$-crossing friend of $K6a2$.}
     \label{fig:K6a2}
 \end{figure}

\begin{ex}\label{ex:Conway}
 For the Conway knot $K=K11n34$, \Cref{alg:friend_search} finds four different friends $K_1,K_2,K_3,K_4$. These are distinguished from $K$ by their volumes, as shown in Table~\ref{tab:volumes}. $K_3$ is the Piccirillo friend of $K$ as constructed in~\cite{Piccirillo_Conway_knot}. The knot $K_4$ shares $4$ surgeries with $K_3$. The knots $K_3$ and $K_4$ are the knots shown in~\cite[Figure 1]{Kegel_Piccirillo} (there for $n=-2$). They can be distinguished by computing their HOMFLYPT polynomials.

 The two friends $K_1$ and $K_2$, shown in \Cref{fig:Conway}, are both new, and we do not have a Kirby calculus proof that they are friends of the Conway knot. We also do not know if they have the same trace as $K$. However, we can compute their $s$-invariants to be non-trivial and check that they are only one crossing change away from a ribbon knot. Thus $K_1$ and $K_2$ both have smooth $4$-genus $g_4=1$.
\end{ex}

\begin{table}[htbp] 
\caption{Volumes, crossings, and smooth $4$-genus $g_4$ of the friends of the Conway knot.}
\label{tab:volumes}
\begin{tabular}{@{}cccc@{}}
\toprule
 knot & volume & crossings & $g_4$ \\
\midrule
  $K11n34$ & $11.21911772538?$ & $11$& $1$\\
  $K_1$ & $13.2204588706?$ & $34$& $1$ \\
  $K_2$ & $14.409961289?$ & $45$& $1$ \\
  $K_3$ & $14.301163110?$ & $55$& $1$ \\
  $K_4$ & $14.301163110?$ & $55$& $1$ \\
\bottomrule
\end{tabular}
\end{table}
    
   \begin{figure}[htbp]
     \centering
     \includegraphics[width=.49\textwidth]{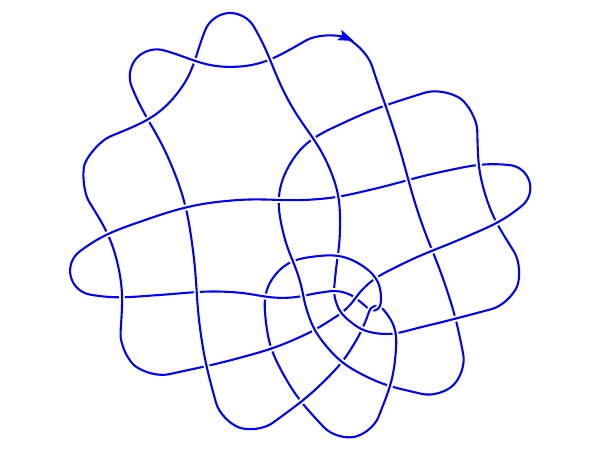}  \includegraphics[width=.5\textwidth]{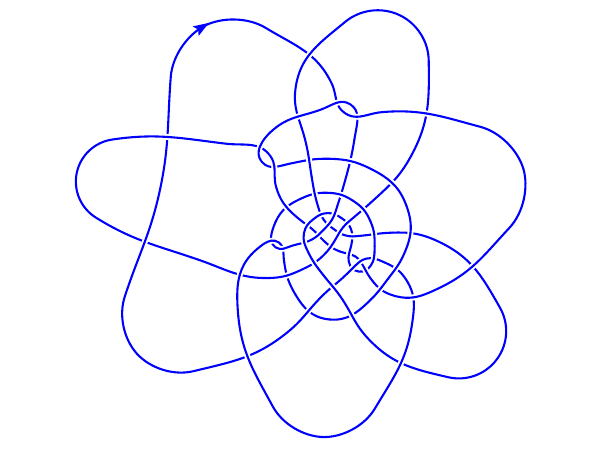}  
     \caption{Two more friends of the Conway knot.}
     \label{fig:Conway}
 \end{figure}
    
\section{A census of friends}\label{sec:census}

In this section, we apply \Cref{alg:friend_search} to the low crossing number knots and the {\em SnapPy} census knots to prove the following.

\begin{thm}\label{thm:census_friends}\hfill
\begin{enumerate}
    \item Among the $2977$ low crossing number knots, at least $1867$ have a friend.
    \item Among the $1267$ {\em SnapPy} census knots, at least $ 183$ have a friend.
\end{enumerate}
These knots are displayed in Tables~\ref{tab:low_crossing} and \ref{tab:census}. PD codes of diagrams of the friends can be found at~\cite{data}.
\end{thm}

\begin{proof}
    The proof goes by running our algorithm. Below, we explain the setup of our experiment resulting in the statement of \Cref{thm:census_friends}.

    Although, in theory, \Cref{alg:friend_search} has an astronomical worst-case running time, it is fast in practice. Still, the running time strongly depends on the number of tetrahedra in the starting triangulation $T$ of the $0$-surgery $K(0)$, and performs well on the low-crossing number knots and the {\em SnapPy} census knots.

    We run our code multiple times over the data sets of the low-crossing number knots and the {\em SnapPy} census knots, each time using \Cref{alg:friend_search} with the exhaustive \texttt{retriangulate} option, the random walk option, as well as the two other methods from \Cref{rem:options}. After each iteration, we remove the knots for which we already found friends and increase the search complexity of the code (e.g., the number of triangulations we try, the bound on the length spectrum, and by letting {\em SnapPy} work harder to enumerate the short slopes). 

    For several knots, we found more than one friend. But in our data, we only list the friend that we found first, which is not necessarily the simplest one.

    Finally, we verify our results. Note that the returned knots $K'$ might in principle be isotopic to $K$. For the verification, we first compute the volumes of $K$ and $K'$. If this fails to distinguish $K$ and $K'$, we compute their Jones or HOMFLYPT polynomials. In a final step, we build the $0$-surgeries $K(0)$ and $K'(0)$ and use {\em SnapPy} and {\em Regina} to verify that they are diffeomorphic. (Note that this last step is only a sanity check, since our algorithm already guarantees this.)
\end{proof}

\begin{rem}
    We do not make any claim about the completeness of the data. It is very likely that there are low-crossing number or census knots that do not appear in our tables, but still admit friends. Our code is merely suggesting that these friends are harder to find.  
\end{rem}

\section{Experiments and running times}\label{sec:runtime}

In this section we give details on running times and a comparison of the performance
of our two variants of \Cref{alg:friend_search}, one based on a breadth first exhaustive
exploration of the Pachner graph (labeled ``exhaustive''), one based on the Markov chain Monte Carlo-style random walk from \cite{Altmann02012026} (labeled ``MCMC''). For completeness, we also compare these algorithms with the methods of drilling out 
``dual curves'' and ``short geodesics'' from \cite{Dunfield_exterior_to_link,Dunfield_Gong}, as described in \Cref{rem:options}.

Interestingly, all four methods have complementary strengths. They efficiently find 
some friends of knots where other methods require longer running times and vice
versa. Moreover, all four methods -- in principle and possibly after some adjustments -- will find every knot friend of a given knot with probability approaching to one as we increase computational resources.\footnote{For the ``short geodesics'' method this is only conjectured and not proven, as discussed in \Cref{rem:options}.}

We compare our methods in two different ways: running times and ``number of curves examined''. Where 
methods are randomised\footnote{Note that all four methods have an element of randomness stemming from the heuristics we use in various steps of our algorithms. We ignore these smaller randomisations and solely focus on significant random effects.}, we report timings and number of attempts from multiple runs. In all of our data, we only compare running times / number of candidates tested for knots that are known to admit friends. Results for unsuccessful friend searches are harder to compare since it is unclear how ``close'' our different methods were to finding a knot friend.

As benchmark for our comparisons, we choose the $47$ knots up to $9$ crossings listed in \Cref{tab:low_crossing}. In our timing runs, we found friends for $46$ of them. In \Cref{tab:comparison} we list which method found which knot.

{\fontsize{9}{11}\selectfont
\begin{table}[htbp] 
\caption{Comparison of the effectiveness of four friend search methods for knots with friends up to $9$ crossings. A cross means a friend was found in at least one of the timing runs.}
\label{tab:comparison}
\begin{tabular}{l|c|c|c|c}
Name & dual curves & short geodesics & exhaustive & MCMC \\
\hline
\hline
K6a1&&&x&x \\
\hline
K6a2&&&x&x \\
\hline
K7a1&&x&x&x \\
\hline
K7a2&&x&x&x \\
\hline
K8a1&x&x&x&x \\
\hline
K8a3&x&x&x&x \\
\hline
K8a4&x&x&x&x \\
\hline
K8a5&x&x&x&x \\
\hline
K8a6&&x&x&x \\
\hline
K8a7&x&x&x&x \\
\hline
K8a9&&x&x&x \\
\hline
K8a10&x&x&x&x \\
\hline
K8a12&&&x&x \\
\hline
K8a14&x&x&x&x \\
\hline
K8a15&x&x&x&x \\
\hline
K8a16&&&x&x \\
\hline
K8a17&&x&x&x \\
\hline
K8n1&&&& \\
\hline
K8n2&&x&x&x \\
\hline
K9a1&x&x&x&x \\
\hline
K9a2&x&x&x&x \\
\hline
K9a3&x&x&x&x \\
\hline
K9a4&x&x&x&x \\
\hline
K9a5&x&x&x&x \\
\hline
K9a6&&x&x&x \\
\hline
K9a7&&x&x&x \\
\hline
K9a8&&x&x& x\\
\hline
K9a10&x&x&x&x \\
\hline
K9a11&x&&x&x \\
\hline
K9a12&x&x&x&x \\
\hline
K9a13&&&x&x \\
\hline
K9a14&x&&x&x \\
\hline
K9a15&&&x&x \\
\hline
K9a17&x&x&&x \\
\hline
K9a18&&x&x&x \\
\hline
K9a21&&x&x&x \\
\hline
K9a22&x&x&x&x \\
\hline
K9a28&x&x&x&x \\
\hline
K9a29&x&&x&x \\
\hline
K9a31&x&&&x \\
\hline
K9a32&x&x&x&x \\
\hline
K9n1&x&x&x&x \\
\hline
K9n2&x&x&x&x \\
\hline
K9n4&&x&x&x \\
\hline
K9n6&&&&x \\
\hline
K9n7&&&x&x \\
\end{tabular}
\end{table}}

As we can see, ``MCMC'' proved to be most effective, with ``dual curves'' lagging behind. Running times of all four methods are compared in \Cref{tab:times}. Here, we can see that ``dual curves'' performs best. 

\begin{figure}[htbp] 
\caption{Running times of all four methods compared. For randomised methods ``exhaustive'' and ``MCMC'', results from multiple runs are presented.}
\label{tab:times}
\includegraphics[width=\textwidth]{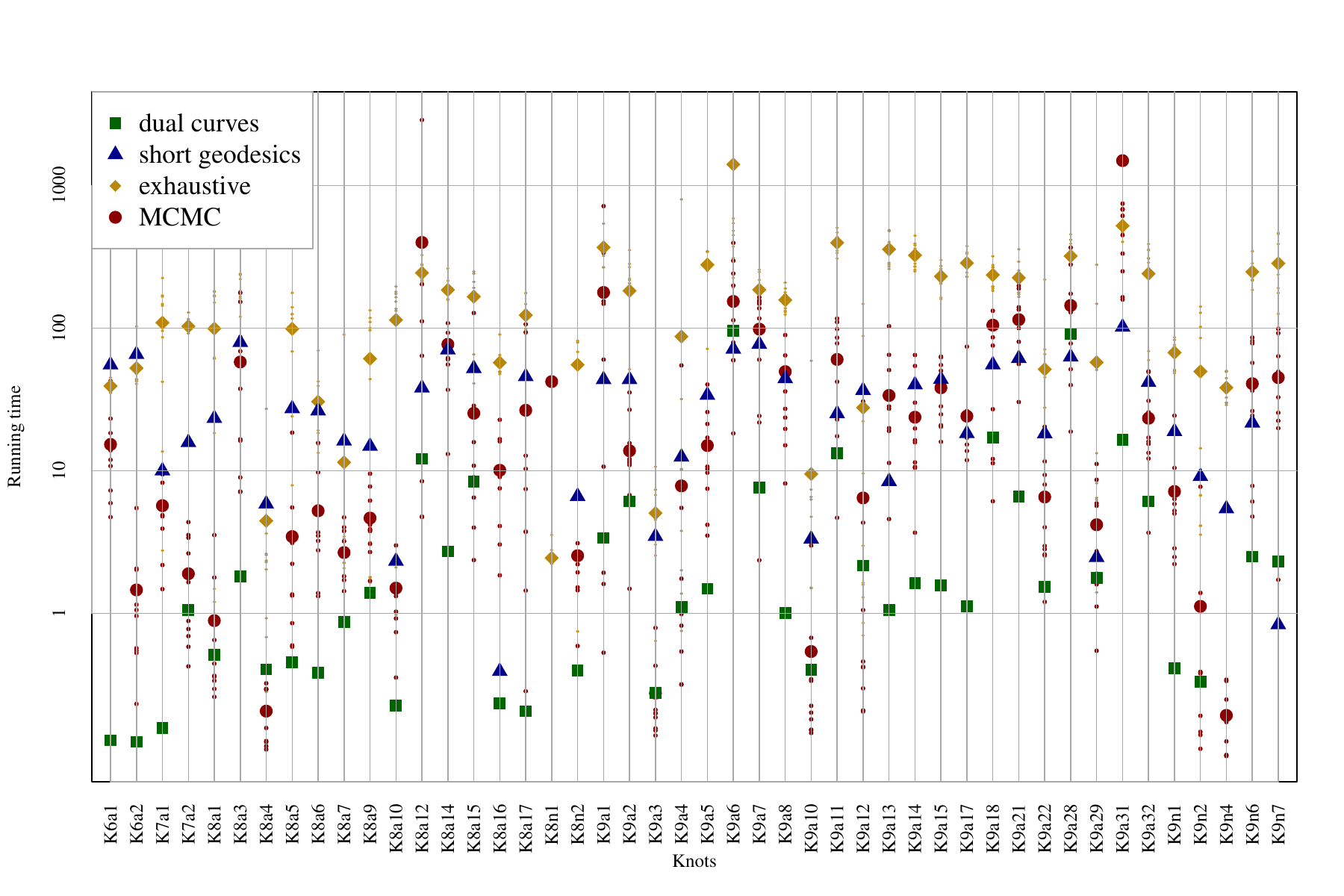}
\end{figure}

In \Cref{tab:candidates}, we compare how many dual curves, geodesics, and edges in triangulations had to be tested, respectively, until a friend was found. While this method does not directly translate into running times, it gives an idea of how efficient our methods search the space of all possible configurations.

\begin{figure}[htbp] 
\caption{Number of candidate curves to check until a friend is found, ranging over all four methods.}
\label{tab:candidates}
\includegraphics[width=\textwidth]{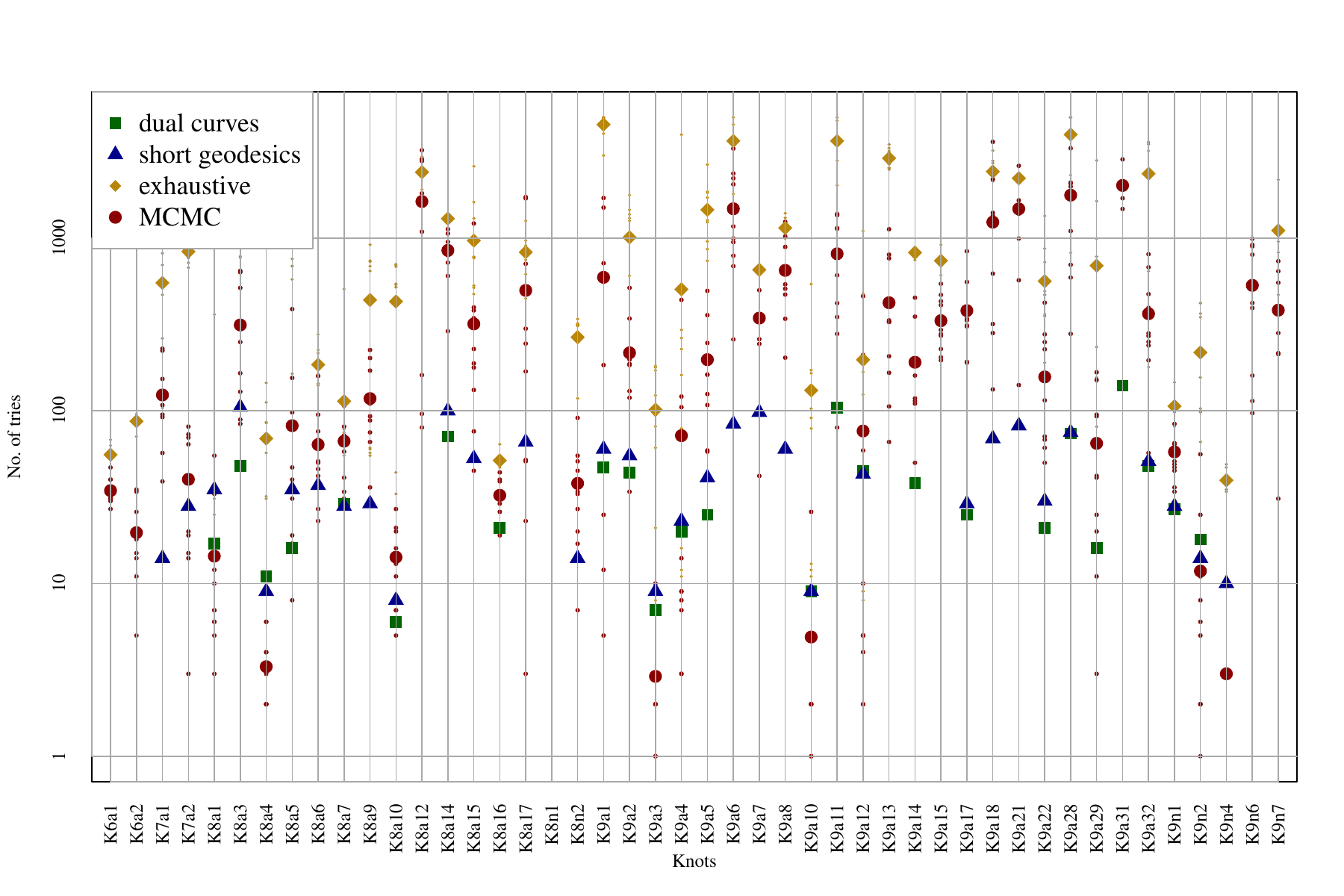}
\end{figure}

\section{Piccirillo friends}\label{sec:Piccirillo}

In order to show that the Conway knot is not slice, Piccirillo describes a process to construct a friend for every knot with unknotting number~$1$.  

\begin{thm} [Piccirillo friend~\cite{Piccirillo_Conway_knot}]\label{thm:Piccirillo}
    Let $K$ be a knot with unknotting number $1$ and let $x$ be an unknotting crossing of $K$. Then there exists a knot $K_x$ such that $K_x$ and $K$ have diffeomorphic $0$-surgeries.\footnote{Note that in some situations it might happen that $K_x$ is isotopic to $K$, see \Cref{prop:Wdouble} below.} 
\end{thm}

\begin{proof}
    From the unknotting crossing $x$ we obtain an ordered $3$-component link $L$ with components $(R,G,B)$ following the procedure shown in \Cref{fig:RGB}. Since $x$ is an unknotting crossing, any two-component sublink of $L$ is a Hopf link. We put surgery coefficients $0$ on $R$ and $G$ and coefficients $\mp 2$ on $B$ ($-2$ if $x$ is a positive crossing and $+2$ if $x$ is negative). Recall that integer surgery on a knot followed by $0$-surgery on a meridian always produces $S^3$. Thus $R(0)G(0)$, $R(0)B(\mp2)$, and $G(0)B(\mp2)$ are all diffeomorphic to $S^3$ and thus the third knot always presents a knot in $S^3$. We call these knots $K_B$, $K_G$ and $K_R$, respectively. Since the homology of the surgered manifold is $\Z$ it follows that all three knots share the same $0$-surgery. By performing explicit handle slides, it is straightforward to check that $K_B$ and $K_R$ are isotopic to $K$. The knots $K_G=K_x$ and $K$ have diffeomorhic $0$-surgeries, and are, in general, not isotopic.
\end{proof}

\begin{figure}[htbp] 
	\centering
	\def\svgwidth{0,9\columnwidth}
	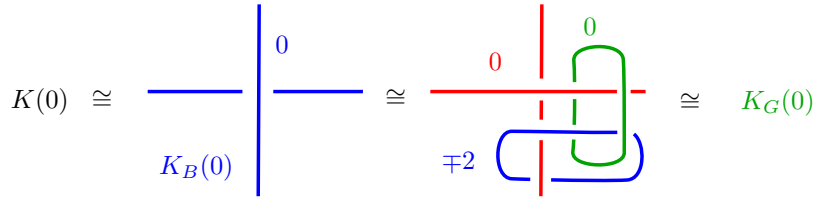
	\caption{Constructing a Piccirillo friend from an unknotting crossing. Note that the red knot is an unknot and thus the surgery on red and blue produces a green knot $K_G$ in $S^3$.}
	\label{fig:RGB}
\end{figure}

\begin{rem}
    For any unknotting crossing, we obtain two Piccirillo friends. Indeed, we can take the same $RGB$ link $L$, but switch the surgery coefficients of $B$ and $G$. Then the same argument as above produces another friend $K'_x$ of $K$. In~\cite{Kegel_Piccirillo} it was observed that $K_x$ and $K'_x$ may be different, but always share $4$ surgeries, cf. $K_3$ and $K_4$ in \Cref{ex:Conway}.
\end{rem}

Note that \Cref{thm:Piccirillo} makes no claim that $K_x$ is different from $K$. In fact, the following proposition describes a general situation in which $K_x$ (and also $K'_x$) are isotopic to $K$.

\begin{prop}\label{prop:Wdouble}
    Let $K$ be a twisted Whitehead double and $x$ an unknotting crossing of $K$. Then the Piccirillo friends $K_x$ and $K'_x$ are isotopic to $K$. 
\end{prop}

\begin{proof}
    Since $K$ is a twisted Whitehead double, it has unknotting number one. In \cite{Unknotting_WHD}, it is shown that any two unknotting crossings of a twisted Whitehead double (that is not the figure eight knot) are equivalent. Thus, we can assume that the unknotting crossing $x$ to construct the Piccirillo friends is one of the two clasp crossings in the Whitehead double construction. (Recall that the figure eight knot admits no friend by~\cite{Gabai}.)

    In \Cref{fig:WHD} we display the $RGB$ link construction for the twisted Whitehead double, which reveals that $K_x$ is again isotopic to $K$. Note that the isotopy from the top left frame to the bottom right frame is only possible if the strip bounding the red unknot has no self-intersection, as for example when it comes from a twisted Whitehead double. (But in general, it will have ribbon self-intersections.) The same construction applies for $K'_x$.
\end{proof}

\begin{figure}[htbp] 
	\centering
	\def\svgwidth{0,9\columnwidth}
	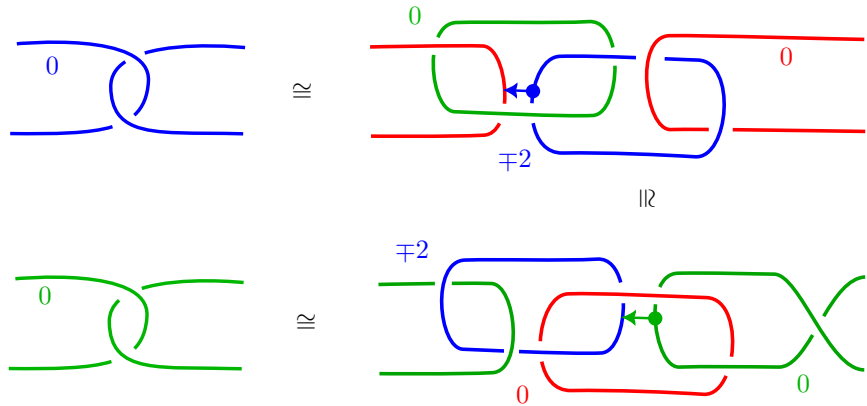
	\caption{The Piccirillo friend of a Whitehead double $K$ is again isotopic to $K$.}
	\label{fig:WHD}
\end{figure}

From~\cite{Piccirillo_Conway_knot}, it follows that the construction of Piccirillo friends is algorithmic. In \cite{Kegel_Weiss}, this was implemented as follows.

\begin{algorithm}
\caption{Piccirillo friend search}
\label{alg:Piccirillo}
\begin{description}
    \item[Input] A diagram of a knot $K$ together with an unknotting crossing $x$.
    \item[Output] Diagrams or triangulations of the complements of Piccirillo friends $K_x$ and $K'_x$.
\end{description}

    \begin{enumerate}
        \item Create a diagram of the $3$-component $RGB$ link $L$ as shown in \Cref{fig:RGB}. 
        \item Use {\em SnapPy} to create a triangulation of the exterior of $L$ and fill it with the correct slopes to obtain triangulations of the knot complements of $K_x$ and $K'_x$. 
        \item Use the heuristic from \cite{Dunfield_exterior_to_link} to create and return diagrams of $K_x$ and $K'_x$. If the heuristic fails, return the triangulations of their knot complements.
    \end{enumerate}
\end{algorithm}

We run this algorithm on the low crossing number knots and the {\em SnapPy} census knots and obtain the following.

\begin{prop}\label{prop:dual_census}\hfill
\begin{enumerate}
    \item $505$ of the low crossing number knots and $84$ of the {\em SnapPy} census knots are known to have unknotting number $1$. 
    \item For all knots in $(1)$ except the $15$ knots listed in \Cref{tab:twist}, there exists an unknotting crossing such that the associated Piccirillo friends are not isotopic to $K$.
    \item The $15$ exceptional knots from \Cref{tab:twist} are all twist knots (and thus twisted Whitehead doubles of the unknot).
\end{enumerate}
\end{prop}

\begin{proof}
    From \cite{KnotInfo} and \cite{CensusKnotInvariants}, we obtain the lists of knots that are known to have unknotting number $1$. We use {\em SnapPy} to perform a search through the Reidemeister graph to find an unknotting crossing $x$ for each of these knots $K$. Next, we use \Cref{alg:Piccirillo} to create Piccirillo friends $K_x$ and $K'_x$. We use verified computations of hyperbolic volumes and HOMFLYPT polynomials to distinguish $K$ from $K_x$ and $K'_x$. This works for all but the $15$ knots listed in \Cref{tab:twist}. We create diagrams for these $15$ knots, from which we can verify that they are all twist knots. Thus, by \Cref{prop:Wdouble} it follows that $K_x$ and $K'_x$ are isotopic to $K$.
\end{proof}

\Cref{prop:Wdouble} and \ref{prop:dual_census} motivate the following conjecture.

\begin{con}\label{conj:WHD}
Let $K$ be a knot with unknotting number one. Then the following are equivalent.
\begin{enumerate}
    \item $K$ is a twisted Whitehead double.
    \item For every unknotting crossing $x$ of $K$, the Piccirillo friends $K_x$ and $K'_x$ are isotopic to $K$.
    \item There exists an unknotting crossing $x$ of $K$ such that the Piccirillo friends $K_x$ and $K'_x$ are isotopic to $K$.
\end{enumerate}
\end{con}

\section{The concordance friend search}\label{sec:concordance_friends}

In this section, we explain the concordance search and the concordance friend search in detail. Moreover, we discuss how we can use this to search for a pair of friends with different sliceness statuses. 

\subsection{The concordance search}

We start with the following definition.

\begin{defi}(Ribbon band)\label{def:Ribbon_band}
    We call the $2$-stranded tangle $R$ shown on the right of \Cref{fig:ribbon_band} a \textit{ribbon band}. 
    A ribbon band attachment to a knot $K$ is the following local operation. We choose a $3$-ball $B^3$ in $S^3$ that intersects the knot $K$ in a trivial $2$-stranded tangle $T$ as shown on the left of \Cref{fig:ribbon_band}. Then we replace $T$ by $R$ to obtain a new knot $K'$. We say that $K'$ is obtained from $K$ by a \textit{ribbon band attachment}.
\end{defi}

 \begin{figure}[htbp]
     \centering
     \includegraphics[width=.7\textwidth]{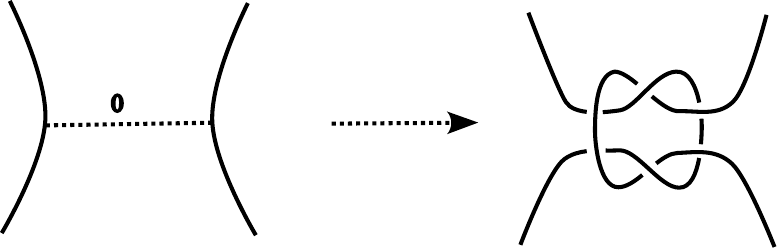}  
     \caption{Attaching a ribbon band to create a concordant knot.}
     \label{fig:ribbon_band}
 \end{figure}

 \begin{cor}\label{cor:ribbon}
     If $K'$ is obtained from $K$ by a ribbon band attachment, then $K$ and $K'$ are concordant. In particular, they have the same $4$-genus.
 \end{cor}

 \begin{proof}
    On the right of \Cref{fig:ribbon_band}, we see the ribbon band $R$ in a $3$-ball $B$. In the figure, we also observe that there are two embedded disks $D_1$ and $D_2$, such that the boundary of each disk consists of two arcs, one arc on the boundary of $B$ and the other coinciding with a component of $R$. These two disks intersect only in their interiors in two ribbon singularities.

    In the product $S^3 \times [0,1]$, we can resolve these two ribbon singularities and obtain two embedded disks. We attach these to the trivial product cobordism 
\[
K \times [0,1] \subset S^3 \times [0,1]
\]
and obtain a ribbon concordance between $K$ and $K'$.
 \end{proof}

This corollary yields the concordance search.

\begin{algorithm}
\caption{Concordance search}
\label{prop:algo2}
\begin{description}
    \item[Input] A diagram $D$ of a knot $K$. Integers $n$ and $t$.
    \item[Output] A list of at most $n$ knots $K'$, presented as diagrams, that are concordant to $K$.
\end{description}

    \begin{enumerate}
        \item Create an empty set $C$ that contains knots concordant to $K$.
        \item The diagram $D$ decomposes the plane into regions. For each region $G$, each pair of strands $s_1$ and $s_2$ of $D$ in the boundary of $G$, and each 
        \[k\in\{-t,-t+1,\ldots,t-1,t\}\]
        perform the following steps.
        \begin{enumerate}
            \item Attach a ribbon band with $k$ half-twists in the region $G$ between the strands $s_1$ and $s_2$, to create a new knot diagram $D'$ of a knot $K'$ that is concordant to $K$, by \Cref{cor:ribbon}.
            \item Use {\em SnapPy} to check if $K'$ is isotopic to $K$ or already contained in $C$. If not, add $K'$ to $C$.
            \item If $C$ contains $n$ elements, return $C$.
        \end{enumerate}
        \item If after Step $(3)$, $C$ still contains fewer than $n$ elements, perform random Reidemeiester moves on the diagram $D$ to obtain a new diagram and perform Step $(3)$ again with this diagram.
    \end{enumerate}
\end{algorithm}

\subsection{The concordance friend search}

Combining \Cref{alg:friend_search} with \Cref{prop:algo2} yields the concordance friend search.

\begin{algorithm}
\caption{Concordance friend search}
\label{algo:concordance_friend_search}
\begin{description}
    \item[Input] A diagram $D$ of a knot $K$. Integers $n$ and $t$.
    \item[Output] A list of at most $n$ concordance friends of $K$ (not obviously concordant to $K$).
\end{description}
    \begin{enumerate}
        \item Create two empty sets $C$ $($containing knots concordant to $K$$)$ and $F$ $($containing concordance friends of $K$$)$.
        \item Run \Cref{prop:algo2}, the concordance search, with parameters $D$, $n$, and $t$, to create $n$ different knots that are concordant to $K$ and not contained in $C$. Add these knots to $C$.
        \item For each $K'$ in $C$:
        \begin{enumerate}
            \item Run \Cref{alg:friend_search}, the friend search.
            \item If this returns a friend $K''$ of $K'$, and if $K''$ is different from $K$, and not contained in $C$ or $F$, add $K''$ to $F$.
            \item If $F$ contains $n$ knots, return $F$.
        \end{enumerate}
        \item If after Step $(4)$, $F$ contains fewer than $n$ elements, either start again at Step~$(3)$ or return $F$.
    \end{enumerate}
\end{algorithm}

\begin{ex}\label{ex:conc_freinds}
   While the friend search was not successful for any non-hyperbolic knot, we were able to find concordance friends in all examples we tried. This is expected by Theorem~1.2 in \cite{BKM_traces}. We discuss some of the interesting examples here.
    \begin{enumerate}
        \item We denote by $W_+$ the positive Whitehead double of the right-handed trefoil, as shown on the left of \Cref{fig_Wp}. Since $W_+$ has trivial Alexander polynomial, it follows that $W_+$ is topologically slice~\cite{Freedman}. On the other hand, it is known that $W_+$ is not smoothly slice, see for example~\cite{Hedden_Whitehead}.
        
        \Cref{algo:concordance_friend_search} shows that the $262$-crossing knot $W_+''$ from \Cref{fig_Wp_friend} is a concordance friend of $W_+$. Indeed, the knot $W_+'$ shown on the right of \Cref{fig_Wp} is obtained from $W_+$ by attaching a single ribbon band, and with {\em SnapPy} it is straightforward to verify that $W_+'$ and $W_+''$ share the same $0$-surgery. We do not know if $W_+''$ is smoothly slice.
        
        \item We denote by $W_-$ the negative Whitehead double of the right-handed trefoil, as shown on the left of \Cref{fig_Wn}. Again, it follows from~\cite{Freedman} that $W_-$ is topologically slice since it has trivial Alexander polynomial. Currently, it is unknown if $W_-$ is smoothly slice.
        
        \Cref{algo:concordance_friend_search} reveals that $W_-$ is concordant to the knot $W_-'$ shown on the right of \Cref{fig_Wn}, which is a friend of the $322$-crossing knot $W_-''$ shown in \Cref{fig_Wn_friend}. Thus $W_-''$ is a concordance friend of $W_-$. We do not know if $W_-''$ is smoothly slice.
        
        \item We write $C$ for the $(2,1)$-cable of the figure eight knot, see \Cref{fig_cable}. $C$ is topologically slice, but not smoothly slice~\cite{cable_fig_8}. 

        With \Cref{algo:concordance_friend_search}, we find the concordance friend $C''$ with $114$-crossings shown in \Cref{fig_Wp_friend}. We do not know if $C''$ is smoothly slice or not.
    \end{enumerate}
\end{ex}

\begin{figure}[htbp]
     \centering
     \includegraphics[width=.9\textwidth] {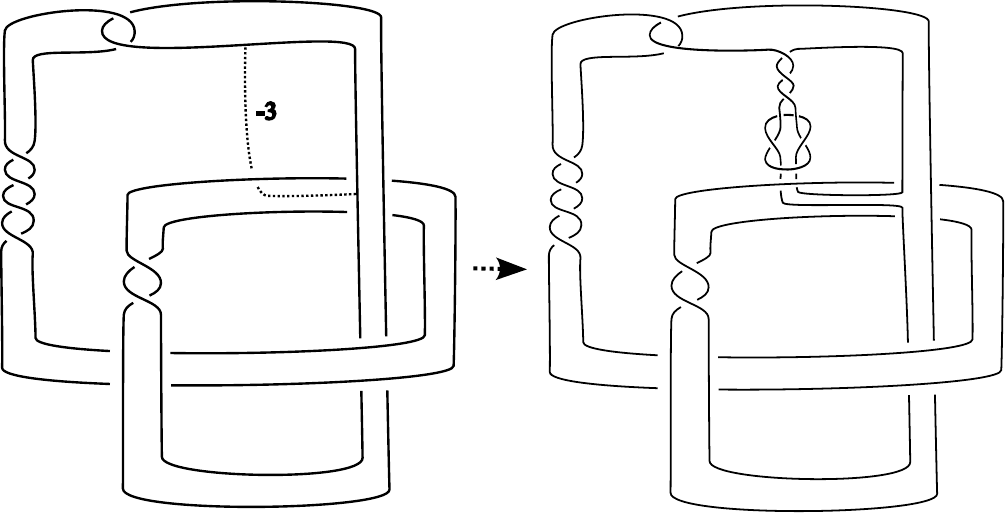}  
     \caption{Left: The positive Whitehead double $W_+$ of the right-handed trefoil. Right: A knot $W_+'$ concordant to $W_+$.}
     \label{fig_Wp}
 \end{figure}

\begin{figure}[htbp]
     \centering
     \includegraphics[width=.92\textwidth] {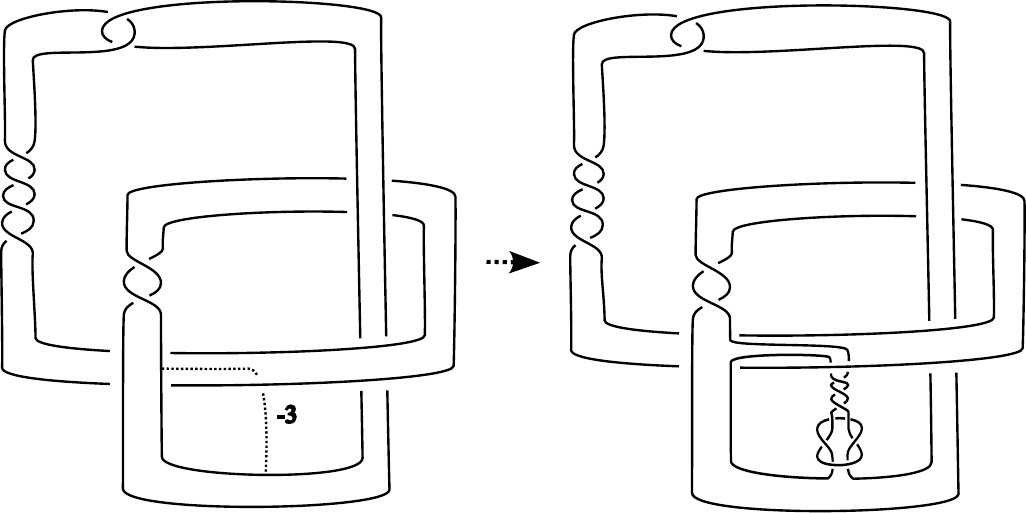}  
     \caption{Left: The negative Whitehead double $W_-$ of the right-handed trefoil. Right: A knot $W_-'$ concordant to $W_-$.}
     \label{fig_Wn}
 \end{figure}

 \begin{figure}[htbp]
     \centering
     \includegraphics[width=.99\textwidth] {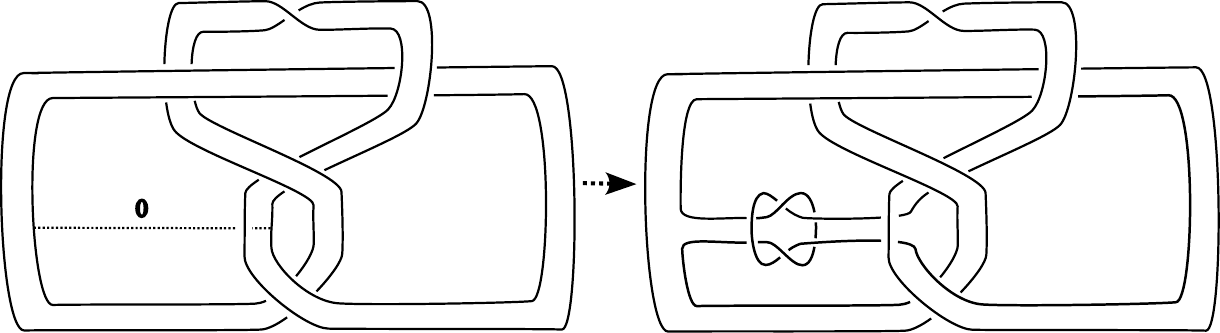}  
     \caption{Left: The $(2,1)$-cable $C$ of the figure eight. Right: A knot $C'$ concordant to $C$.}
     \label{fig_cable}
 \end{figure}

  \begin{figure}[htbp]
     \centering
     \includegraphics[width=.498\textwidth] {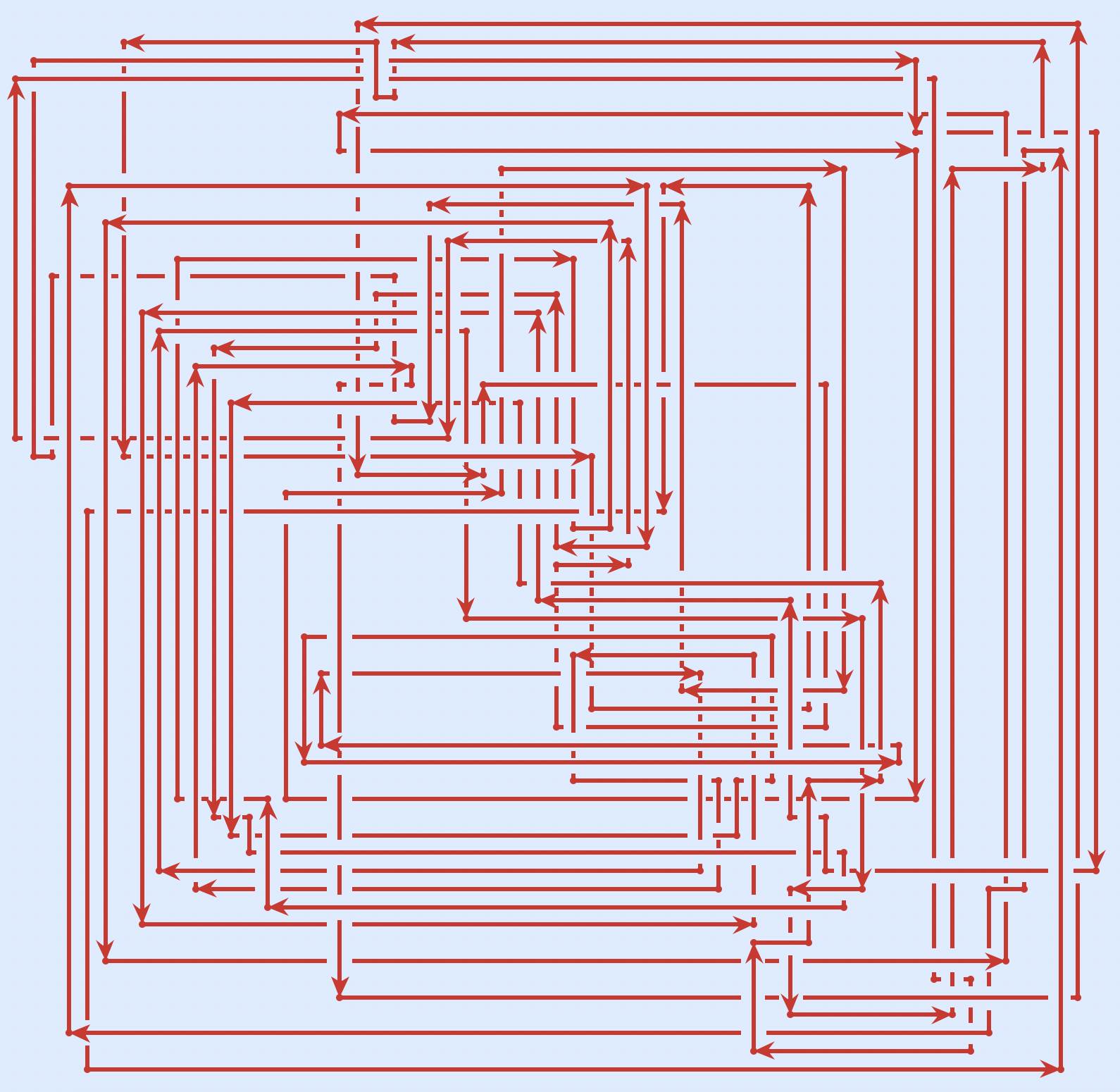}  
     \includegraphics[width=.49\textwidth]{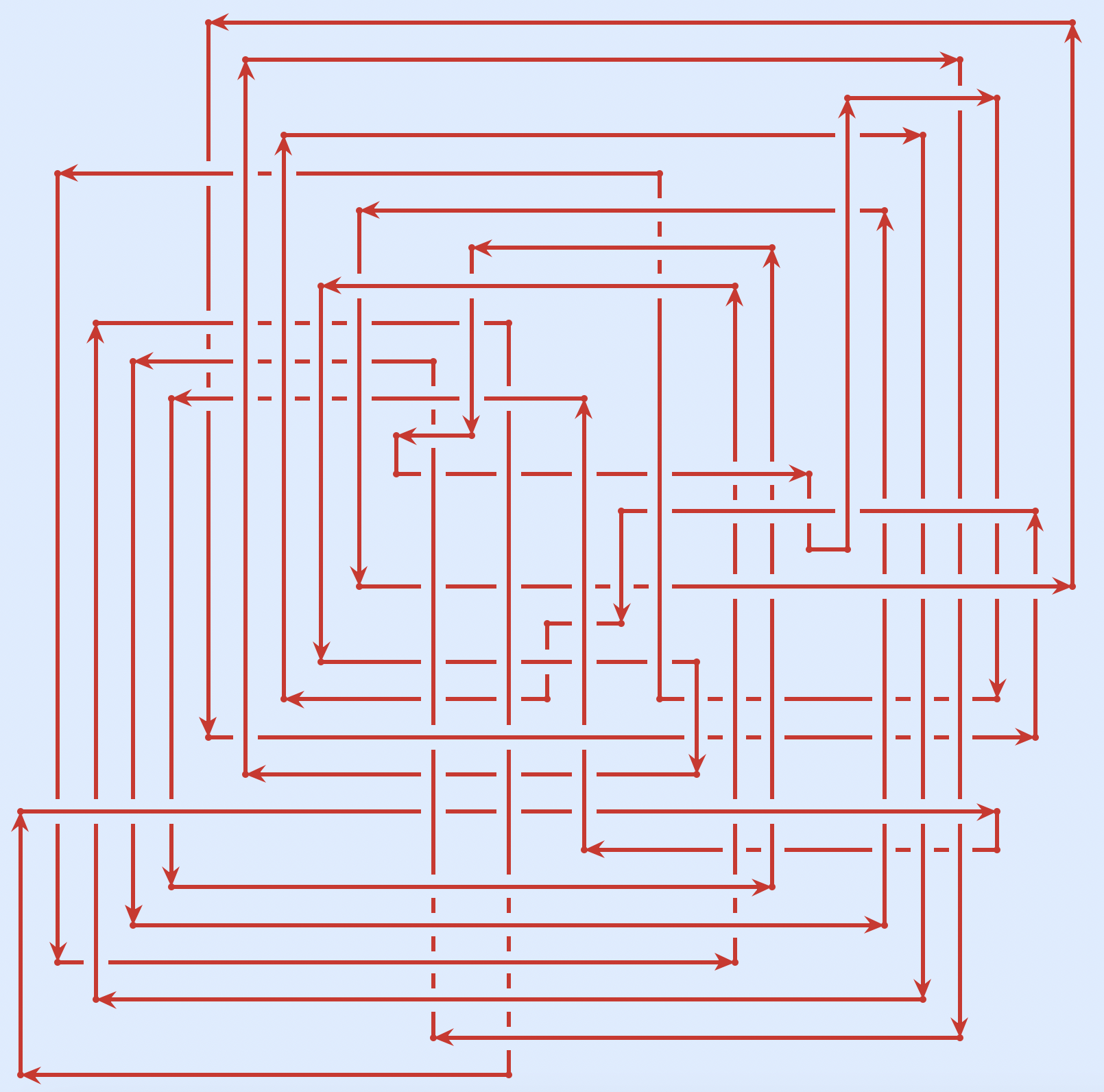}  
     \caption{Concordance friends of the positive Whitehead double of the right-handed trefoil (left) and of the $(2,1)$-cable of the figure eight (right).}
     \label{fig_Wp_friend}
 \end{figure}

  \begin{figure}[htbp]
     \centering
     \includegraphics[width=.75\textwidth]{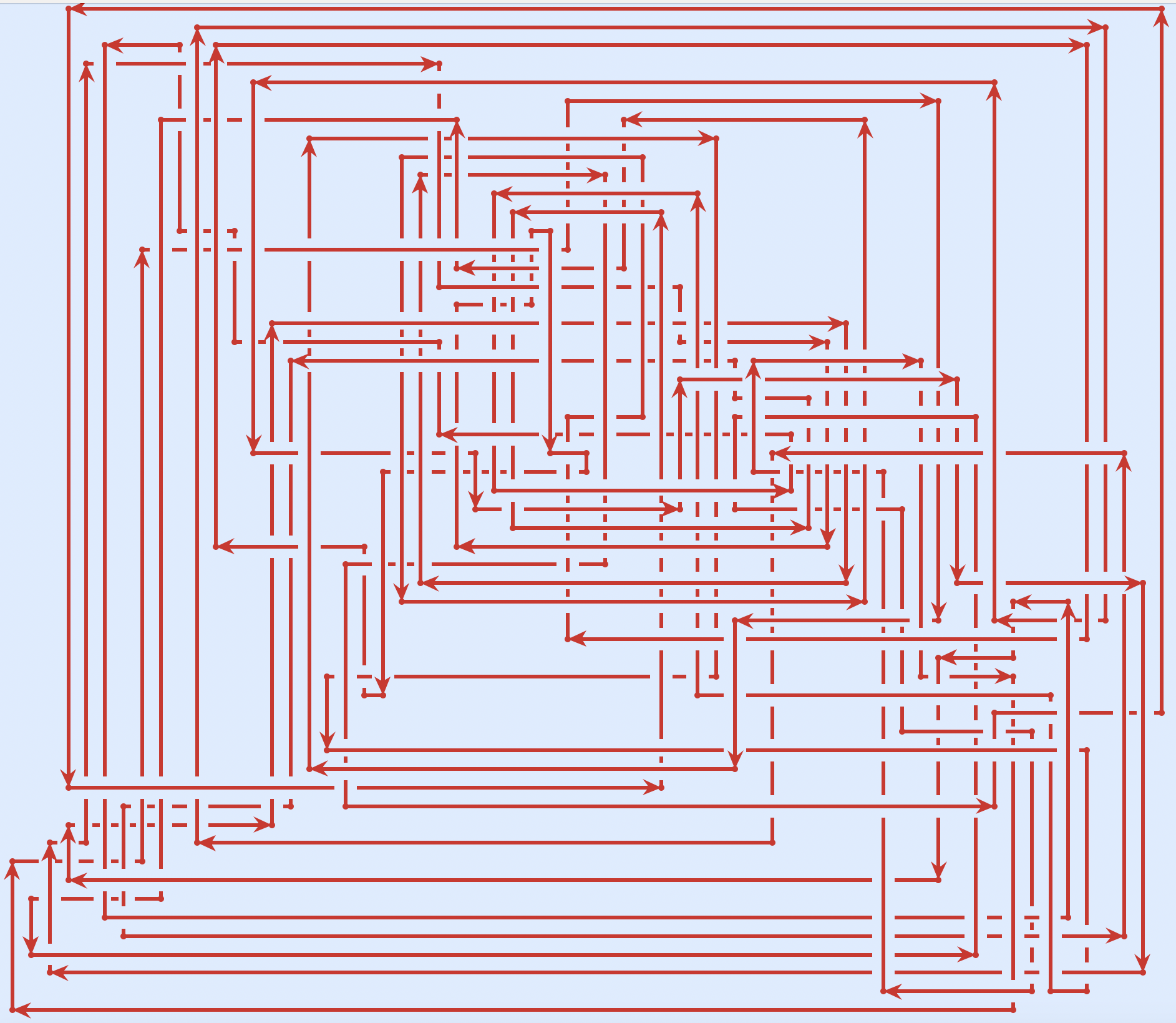}  
     \caption{A concordance friend of the negative Whitehead double of the right-handed trefoil.}
     \label{fig_Wn_friend}
 \end{figure}

\subsection{Search for friends with different sliceness statuses}

By combining the friend search (\Cref{alg:friend_search}) and the concordance search (\Cref{prop:algo2}), we have a potential way to search for concordance friends with different sliceness statuses.

For that, we start with a finite set $F_1$ of slice knots and use \Cref{prop:algo2} to create a new set $F_2$ of knots concordant to a knot in $F_1$. (By construction, the knots in $F_2$ are all slice.) Then we use \Cref{alg:friend_search} to create a set $F_3$ of knots that are friends with a knot in $F_2$. Then we iterate this process to create an arbitrarily long chain of knots that are concordance friends with a slice knot. If any computable slice obstruction for a knot in this chain is non-vanishing, then there exists a pair of concordance friends with different sliceness status. Since the topological Poincaré conjecture in dimension $4$ is true~\cite{Freedman}, two concordance friends have necessarily the same topological sliceness status. Thus, the above sliceness obstructions need to be smooth obstructions. Example candidates for such smooth sliceness obstructions are the s-invariant and its generalizations, or potentially also the sliceness obstructions coming from knot Floer homology. In principle, these obstructions are computable, but quickly become infeasible with increasing crossing number. Since the knots created with \Cref{alg:friend_search} and \ref{prop:algo2} usually have a high crossing number, this approach is not very likely to terminate.

On the other hand, it is possible to find a ribbon certificate even for knots of large crossing number. In \cite{ribbons_with_ML}, modern machine learning methods are used for detecting a certificate for a knot bounding a ribbon disk. While no decidable algorithm for ribbonness is known, the machine learning approach produces a \textit{verifiable certificate} for ribbonness if successful, and the authors demonstrate that their programs can reliably detect ribbon certificates even for knots of high crossing number. A different algorithm is described in~\cite{Dunfield_Gong}.

\begin{algorithm}
\caption{Ribbon certificate search \cite{ribbons_with_ML,Dunfield_Gong}}
\label{alg:ml_ribbon_search}
\begin{description}
    \item[Input] A diagram of a knot $K$.
    \item[Output] Upon success, a verifiable certificate that $K$ is ribbon and hence slice. Otherwise, return \texttt{fail}.
\end{description}
\begin{enumerate}
    \item Using the input diagram, search for a sequence of diagrammatic band additions and Reidemeister moves producing a ribbon presentation for $K$.
    \item If successful, return the ribbon presentation of $K$.
\end{enumerate}
\end{algorithm}

Instead of searching for a knot that is not slice, we propose the following search.

\begin{algorithm}
\caption{Search for a pair of knots that are concordance friends but have different sliceness statuses.}
\label{algo:conc_friends}
\begin{description}
    \item[Input] A finite set $F_1$ of diagrams of knots that are topologically slice but not smoothly slice.
    \item[Output] Upon success, a slice knot which is a concordant friend with a knot in $F_1$. 
\end{description}
   \begin{enumerate}
        \item Use \Cref{prop:algo2} to create a finite set $F_2$ of knots that are concordant to a knot in $F_1$.
        \item Use \Cref{prop:algo2} to create a finite set $F_3$ of knots that are friends with a knot in $F_2$.
        \item For each knot $K$ in $F_2$, create a diagram using \cite{Dunfield_exterior_to_link}, and check if \Cref{alg:ml_ribbon_search} identifies $K$ as ribbon. If so, we have found a pair of concordance friends with different sliceness statuses, and we return them.
        \item Iterate this process, i.e.\ use $F_3$ as input and go to Step (1).
    \end{enumerate}
\end{algorithm}

Examples of knots that are topologically slice but not smoothly slice are the $(2,1)$-cable of the figure eight knot~\cite{cable_fig_8}, the positive Whitehead double of the right-handed trefoil~\cite{Hedden_Whitehead}, cf. \Cref{ex:conc_freinds}, and the Conway knot and its friends~\cite{Piccirillo_Conway_knot}, cf. \Cref{ex:Conway}. 

Steps $(1)$ and $(2)$ of \Cref{algo:conc_friends}, and the ribbon certificate search (\Cref{alg:ml_ribbon_search}) are reasonable fast in practice. The bottleneck turns out to be the creation of the knot diagrams in Step $(3)$ of \Cref{algo:conc_friends}. The friends constructed in the friend search (\Cref{alg:friend_search}) come just as a triangulation of the knot complement together with an $S^3$-filling. To run the ribbon certificate search and the concordance search, we need diagrams. Since the friends constructed with the friend search often have diagrams with several thousand crossings, creating the diagrams takes a long time and often does not terminate at all. In practice, it was most successful to start with the Conway knot and its $4$ friends from \Cref{ex:Conway} since these are hyperbolic with hyperbolic $0$-surgery (and our code works best for hyperbolic knots), and have reasonably small complexity. We discuss this in more detail in the next section. Other simple examples of hyperbolic, topological slice knots that are not smoothly slice are given in~\cite{Dunfield_Gong}. These are the knots $K16n68278$, $17nh0010647$, and $18nh00098198$.

\subsection{The Conway knot and its friends}

We run \Cref{algo:conc_friends} for the Conway knot to search for a pair of friends with different sliceness statuses. In detail, we perform the following steps.

\begin{enumerate}
    \item Set $F_1$ to be the set consisting of the Conway knot $K11n34$ and the three knots $K16n68278$, $17nh0010647$, and $18nh00098198$ from~\cite{Dunfield_Gong}. All four knots are topologically slice, but admit a friend that is not smoothly slice \cite{Piccirillo_Conway_knot,Dunfield_Gong}. We search for a concordance friend of a knot in $F_1$.
    \item For each knot in $F_1$, use the friend search (\Cref{alg:friend_search}) to search for friends. In total we find $13$ friends, including the four friends $K_1,K_2,K_3,K_4$ from \Cref{ex:Conway}. Define $F_2$ to be the set of the four knots from $F_1$ and their $13$ friends.
    \item For each knot $K$ in $F_2$, use \Cref{algo:conc_friends} to create knot diagrams presenting a knot concordant to $K$. In total, we create $5196$ such diagrams. We use {\em SnapPy} to check that these $5196$ knot diagrams present $1950$ non-isotopic knots, which we combine into a set $F_3$. On a standard laptop, this step takes less than $1$ minute and could easily be expanded.
    \item For every knot $K$ in $F_3$, run the friend search (\Cref{alg:friend_search}). This identifies $201$ exteriors of friends. Since this step consists of $1950$ steps that are completely independent of each other, we can easily parallelize this step. We use multiprocessing to run the friend search on $32$ parallel examples at the same time. This step takes around one hour on a standard laptop. 
    \item Use~\cite{Dunfield_exterior_to_link} to create diagrams of the $201$ friends. This step takes the most time. For $17$ of the friends, no diagram was found. Combine the remaining $184$ diagrams into a set $F_4$. Each element in $F_4$ represents a concordance friend of a knot in $F_1$. The smallest diagram in $F_4$ has $79$ and the largest $1620$ crossings. The average crossing number of the diagrams is $305$, and the median is $268$.
    \item For each of the $184$ diagrams of the concordance friends of the Conway knot, we use \Cref{alg:ml_ribbon_search} to (unsuccessfully) search for a ribbon certificate. For every knot, we attach $30,000$ bands. This step takes $10-60$ minutes per diagram. 
    \item Iterate this process: For each knot $K$ in $F_4$, we create more diagrams of knots concordant to $K$. In total, we create $1725$ such diagrams, yielding $1414$ different knots that are not already contained in $F_1$, $F_2$, $F_3$, or $F_4$. Let $F_5$ be the set containing these $1414$ knots.
    \item For the knots in $F_5$, we run the friend search. This identifies $129$ friends. For $78$ of these, we were able to identify a diagram, but no ribbon certificate. Combine these knots into the set $F_6$. The smallest example in $F_6$ has $91$ and the largest $1497$ crossings. The average crossing number of the diagrams is $344$, and the median is $223$.
\end{enumerate}

\begin{figure}[htbp]
     \centering
     \includegraphics[width=.479\textwidth] {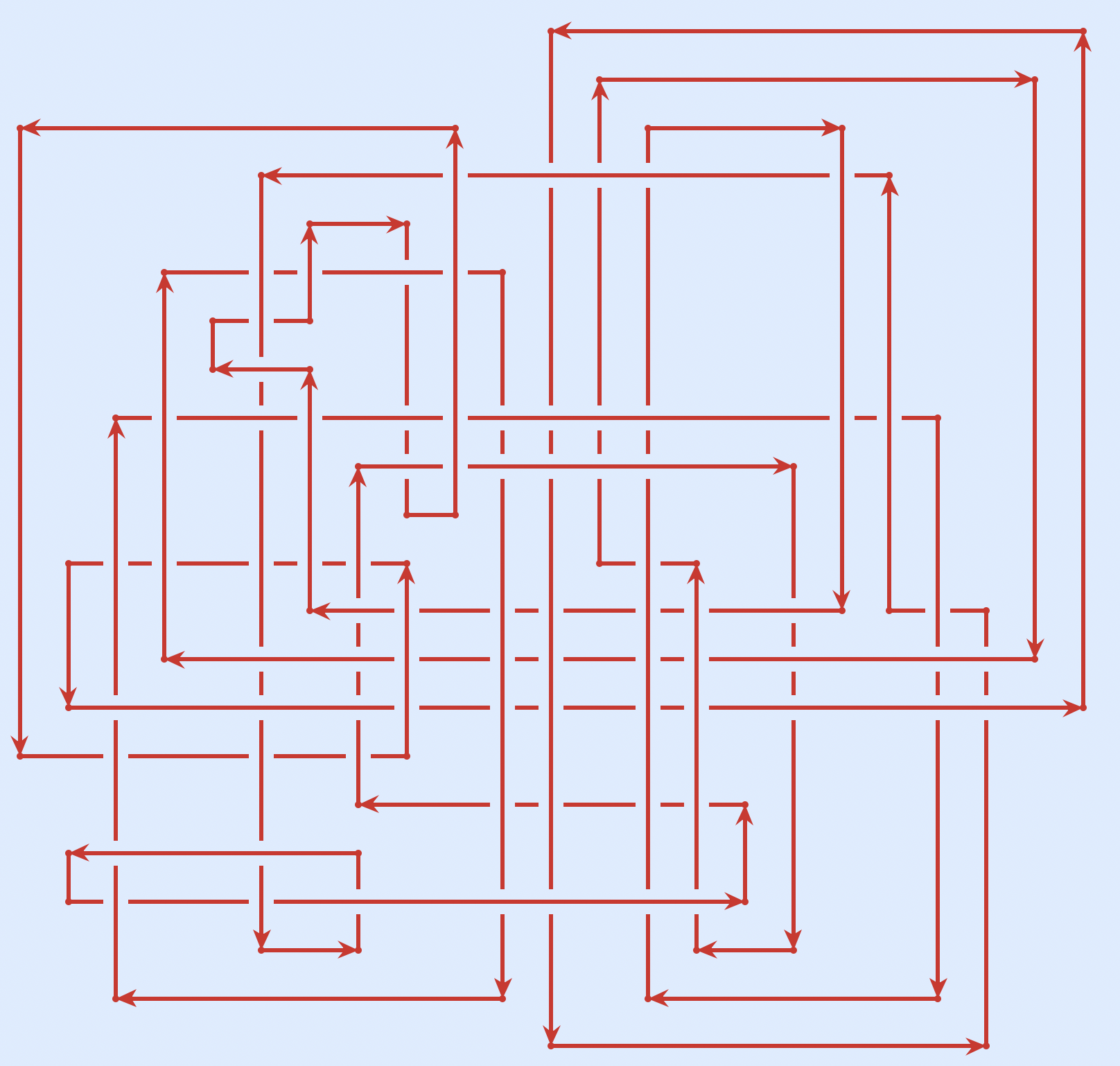}  
     \includegraphics[width=.511\textwidth]{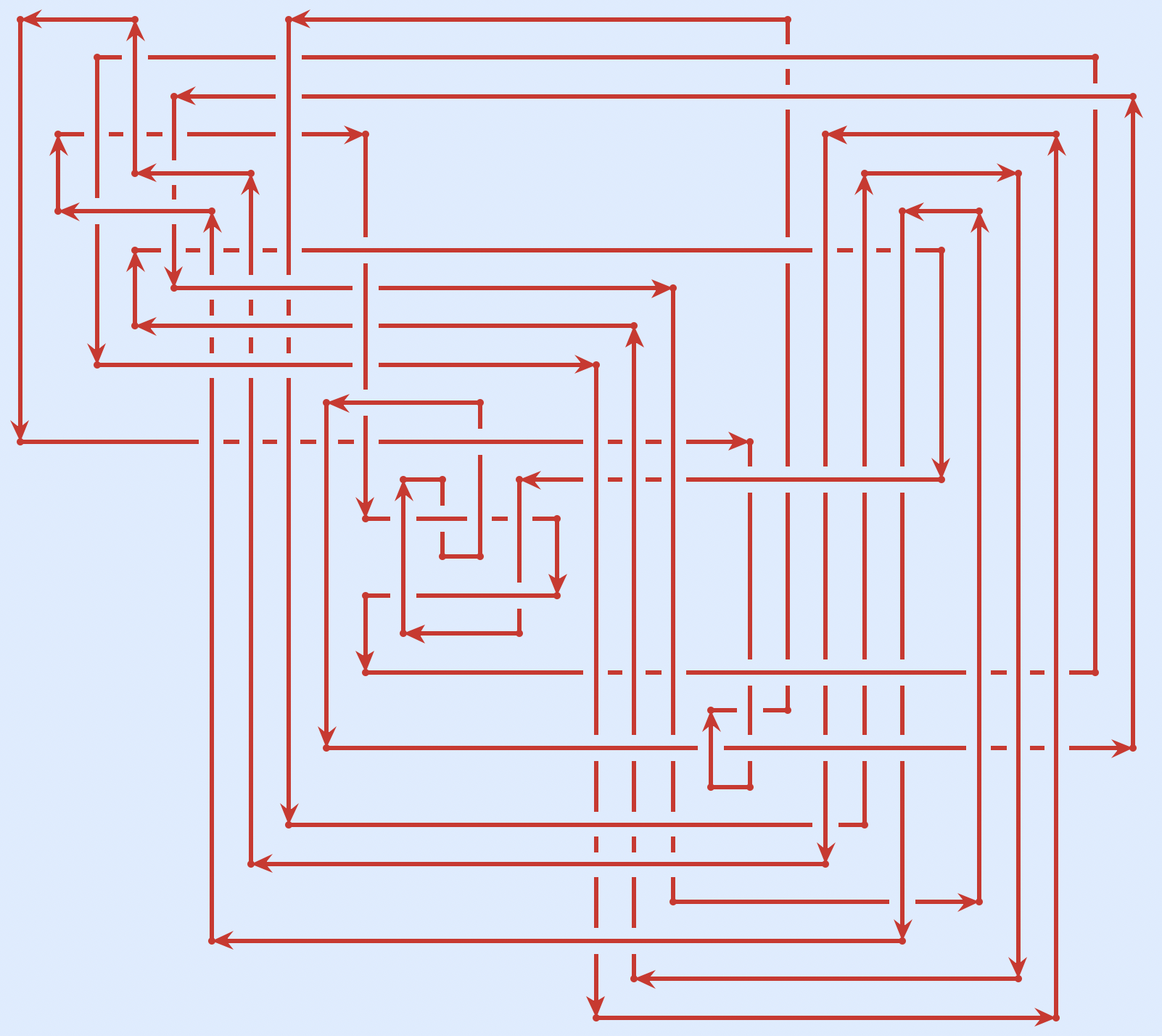}  
     \caption{The two concordance friends with minimal crossing number from $F_4$ (left) and $F_6$ (right).}
     \label{fig:conc_friends}
 \end{figure}

In summary, we created $3643$ diagrams that belong to concordance friends of a topological but not smoothly slice knot. The two examples from $F_4$ and $F_6$ with minimal crossing numbers are displayed in \Cref{fig:conc_friends}. The full list can be accessed at~\cite{data}. For none of these examples we were able to identify a ribbon certificate. If it turns out that one of these examples is slice, then there exists an exotic $4$-sphere. 

{\fontsize{6}{9}\selectfont
\begin{longtable}{@{}lllllllll@{}}
\caption{Low crossing number knots that admit a friend.} \label{tab:low_crossing} \\

\hline
\endfirsthead

\multicolumn{9}{c}%
{{\tablename\ \thetable{} -- continued from previous page}} \\
\hline
\endhead

\hline \multicolumn{9}{r}{{Continued on next page}} \\ 
\endfoot

\hline \hline
\endlastfoot

  $K6a1$ &  
  $K6a2$ &  
  $K7a1$ &  
  $K7a2$ &  
  $K8a1$ &  
  $K8a3$ &
  $K8a4$ &  
  $K8a5$ &  
  $K8a6$ \\ 
  $K8a7$ &  
  $K8a9$ &  
  $K8a10$ &
  $K8a12$ &  
  $K8a14$ &  
  $K8a15$ &  
  $K8a16$ &  
  $K8a17$ &  
  $K8n1$ \\  
  $K8n2$ & 
  $K9a1$ &  
  $K9a2$ &  
  $K9a3$ &  
  $K9a4$ &  
  $K9a5$ &  
  $K9a6$ &  
  $K9a7$ &  
  $K9a8$ \\  
  $K9a10$ &  
  $K9a11$ &  
  $K9a12$ &  
  $K9a13$ &  
  $K9a14$ &  
  $K9a15$ &  
  $K9a17$ &  
  $K9a18$ &  
  $K9a21$ \\  
  $K9a22$ &  
  $K9a28$ &  
  $K9a29$ &  
  $K9a31$ &  
  $K9a32$ &  
  $K9n1$ &  
  $K9n2$ &  
  $K9n4$ &  
  $K9n6$ \\   
  $K9n7$ &  
  $K10a1$ &  
  $K10a2$ &  
  $K10a3$ &  
  $K10a5$ &  
  $K10a6$ &  
  $K10a7$ &  
  $K10a10$ &  
  $K10a11$\\   
  $K10a12$ &  
  $K10a15$ &  
  $K10a16$ &  
  $K10a18$ &  
  $K10a19$ &  
  $K10a20$ &  
  $K10a21$ &  
  $K10a23$ &  
  $K10a24$ \\   
  $K10a25$ &  
  $K10a27$ &  
  $K10a29$ &  
  $K10a30$ &  
  $K10a31$ &  
  $K10a32$ &  
  $K10a34$ &  
  $K10a35$ &  
  $K10a36$ \\   
  $K10a37$ &  
  $K10a38$ &  
  $K10a39$ &  
  $K10a42$ &  
  $K10a43$ &  
  $K10a44$ &  
  $K10a47$ &  
  $K10a48$ &  
  $K10a49$ \\   
  $K10a50$ &  
  $K10a52$ &  
  $K10a53$ &  
  $K10a54$ &  
  $K10a55$ &  
  $K10a57$ &  
  $K10a58$ &  
  $K10a62$ &  
  $K10a63$ \\   
  $K10a64$ &  
  $K10a65$ &  
  $K10a66$ &  
  $K10a67$ &  
  $K10a68$ &  
  $K10a69$ &  
  $K10a71$ &  
  $K10a72$ &  
  $K10a74$ \\   
  $K10a76$ &  
  $K10a77$ &  
  $K10a78$ &  
  $K10a79$ &  
  $K10a80$ &  
  $K10a82$ &  
  $K10a83$ &  
  $K10a84$ &  
  $K10a85$ \\   
  $K10a86$ &  
  $K10a87$ &  
  $K10a88$ &  
  $K10a89$ &  
  $K10a91$ &  
  $K10a92$ &  
  $K10a93$ &  
  $K10a94$ &  
  $K10a95$ \\   
  $K10a97$ &  
  $K10a99$ &  
  $K10a100$ &  
  $K10a103$ &  
  $K10a104$ &  
  $K10a105$ &  
  $K10a106$ &  
  $K10a107$ &  
  $K10a109$ \\   
  $K10a110$ &  
  $K10a111$ &  
  $K10a112$ &  
  $K10a118$ &  
  $K10a119$ &  
  $K10a120$ &  
  $K10a121$ &  
  $K10a122$ &  
  $K10n1$ \\   
  $K10n2$ &  
  $K10n3$ &  
  $K10n4$ &  
  $K10n5$ &  
  $K10n8$ &  
  $K10n9$ &  
  $K10n10$ &  
  $K10n11$ &  
  $K10n12$ \\   
  $K10n13$ &  
  $K10n15$ &  
  $K10n17$ &  
  $K10n18$ &  
  $K10n19$ &  
  $K10n20$ &  
  $K10n23$ &  
  $K10n24$ &  
  $K10n25$ \\   
  $K10n26$ &  
  $K10n28$ &  
  $K10n32$ &  
  $K10n34$ &  
  $K10n35$ &  
  $K10n37$ &  
  $K10n38$ &  
  $K10n39$ &  
  $K10n40$ \\   
  $K10n41$ &  
  $K11a3$ &  
  $K11a4$ &  
  $K11a5$ &  
  $K11a6$ &  
  $K11a7$ &  
  $K11a8$ &  
  $K11a10$ &  
  $K11a11$ \\   
  $K11a12$ &  
  $K11a13$ &  
  $K11a14$ &  
  $K11a16$ &  
  $K11a17$ &  
  $K11a18$ &  
  $K11a19$ &  
  $K11a21$ &  
  $K11a22$ \\   
  $K11a23$ &  
  $K11a24$ &  
  $K11a25$ &  
  $K11a27$ &  
  $K11a28$ &  
  $K11a29$ &  
  $K11a30$ &  
  $K11a33$ &  
  $K11a34$ \\   
  $K11a35$ &  
  $K11a36$ &  
  $K11a37$ &  
  $K11a38$ &  
  $K11a39$ &  
  $K11a41$ &  
  $K11a44$ &  
  $K11a45$ &  
  $K11a46$ \\   
  $K11a47$ &  
  $K11a49$ &  
  $K11a50$ &  
  $K11a52$ &  
  $K11a54$ &  
  $K11a55$ &  
  $K11a56$ &  
  $K11a57$ &  
  $K11a58$ \\   
  $K11a61$ &  
  $K11a65$ &  
  $K11a66$ &  
  $K11a67$ &  
  $K11a69$ &  
  $K11a70$ &  
  $K11a71$ &  
  $K11a72$ &  
  $K11a73$ \\   
  $K11a75$ &  
  $K11a76$ &  
  $K11a77$ &  
  $K11a78$ &  
  $K11a79$ &  
  $K11a80$ &  
  $K11a82$ &  
  $K11a84$ &  
  $K11a85$ \\   
  $K11a86$ &  
  $K11a87$ &  
  $K11a88$ &  
  $K11a89$ &  
  $K11a90$ &  
  $K11a91$ &  
  $K11a92$ &  
  $K11a93$ &  
  $K11a96$ \\   
  $K11a98$ &  
  $K11a101$ &  
  $K11a102$ &  
  $K11a103$ &  
  $K11a104$ &  
  $K11a106$ &  
  $K11a107$ &  
  $K11a108$ &  
  $K11a109$ \\   
  $K11a110$ &  
  $K11a111$ &  
  $K11a112$ &  
  $K11a114$ &  
  $K11a115$ &  
  $K11a116$ &  
  $K11a118$ &  
  $K11a119$ &  
  $K11a121$ \\   
  $K11a122$ &  
  $K11a125$ &  
  $K11a126$ &  
  $K11a128$ &  
  $K11a130$ &  
  $K11a131$ &  
  $K11a132$ &  
  $K11a133$ &  
  $K11a134$ \\   
  $K11a136$ &  
  $K11a137$ &  
  $K11a138$ &  
  $K11a139$ &  
  $K11a141$ &  
  $K11a145$ &  
  $K11a146$ &  
  $K11a147$ &  
  $K11a148$ \\   
  $K11a149$ &  
  $K11a151$ &  
  $K11a152$ &  
  $K11a153$ &  
  $K11a154$ &  
  $K11a156$ &  
  $K11a157$ &  
  $K11a158$ &  
  $K11a159$ \\   
  $K11a160$ &  
  $K11a161$ &  
  $K11a162$ &  
  $K11a163$ &  
  $K11a164$ &  
  $K11a165$ &  
  $K11a166$ &  
  $K11a167$ &  
  $K11a168$ \\   
  $K11a169$ &  
  $K11a172$ &  
  $K11a174$ &  
  $K11a175$ &  
  $K11a176$ &  
  $K11a178$ &  
  $K11a180$ &  
  $K11a181$ &  
  $K11a183$ \\   
  $K11a184$ &  
  $K11a185$ &  
  $K11a187$ &  
  $K11a189$ &  
  $K11a190$ &  
  $K11a193$ &  
  $K11a195$ &  
  $K11a197$ &  
  $K11a198$ \\   
  $K11a199$ &  
  $K11a201$ &  
  $K11a202$ &  
  $K11a209$ &  
  $K11a210$ &  
  $K11a214$ &  
  $K11a216$ &  
  $K11a217$ &  
  $K11a218$ \\   
  $K11a219$ &  
  $K11a221$ &  
  $K11a222$ &  
  $K11a226$ &  
  $K11a228$ &  
  $K11a229$ &  
  $K11a230$ &  
  $K11a231$ &  
  $K11a232$ \\   
  $K11a233$ &  
  $K11a248$ &  
  $K11a249$ &  
  $K11a252$ &  
  $K11a253$ &  
  $K11a254$ &  
  $K11a255$ &  
  $K11a256$ &  
  $K11a257$ \\   
  $K11a258$ &  
  $K11a260$ &  
  $K11a261$ &  
  $K11a262$ &  
  $K11a264$ &  
  $K11a265$ &  
  $K11a266$ &  
  $K11a267$ &  
  $K11a268$ \\   
  $K11a269$ &  
  $K11a270$ &  
  $K11a272$ &  
  $K11a273$ &  
  $K11a275$ &  
  $K11a277$ &  
  $K11a278$ &  
  $K11a279$ &  
  $K11a280$ \\   
  $K11a281$ &  
  $K11a282$ &  
  $K11a283$ &  
  $K11a284$ &  
  $K11a285$ &  
  $K11a287$ &  
  $K11a289$ &  
  $K11a290$ &  
  $K11a294$ \\   
  $K11a296$ &  
  $K11a297$ &  
  $K11a300$ &  
  $K11a303$ &  
  $K11a304$ &  
  $K11a305$ &  
  $K11a307$ &  
  $K11a311$ &  
  $K11a312$ \\   
  $K11a313$ &  
  $K11a315$ &  
  $K11a316$ &  
  $K11a317$ &  
  $K11a322$ &  
  $K11a324$ &  
  $K11a325$ &  
  $K11a326$ &  
  $K11a328$ \\   
  $K11a330$ &  
  $K11a333$ &  
  $K11a345$ &  
  $K11a347$ &  
  $K11a349$ &  
  $K11a350$ &  
  $K11a351$ &  
  $K11n1$ &  
  $K11n3$ \\  
  $K11n4$ &  
  $K11n5$ &  
  $K11n6$ &  
  $K11n7$ &  
  $K11n8$ &  
  $K11n11$ &  
  $K11n12$ &  
  $K11n15$ &  
  $K11n17$ \\   
  $K11n18$ &  
  $K11n20$ &  
  $K11n21$ &  
  $K11n22$ &  
  $K11n23$ &  
  $K11n24$ &  
  $K11n25$ &  
  $K11n26$ &  
  $K11n28$ \\   
  $K11n29$ &  
  $K11n30$ &  
  $K11n31$ &  
  $K11n32$ &  
  $K11n33$ &  
  $K11n34$ &  
  $K11n35$ &  
  $K11n37$ &  
  $K11n38$ \\   
  $K11n39$ &  
  $K11n40$ &  
  $K11n41$ &  
  $K11n42$ &  
  $K11n43$ &  
  $K11n44$ &  
  $K11n45$ &  
  $K11n46$ &  
  $K11n47$ \\   
  $K11n48$ &  
  $K11n49$ &  
  $K11n50$ &  
  $K11n51$ &  
  $K11n52$ &  
  $K11n53$ &  
  $K11n54$ &  
  $K11n55$ &  
  $K11n56$ \\   
  $K11n58$ &  
  $K11n60$ &  
  $K11n61$ &  
  $K11n62$ &  
  $K11n63$ &  
  $K11n64$ &  
  $K11n65$ &  
  $K11n66$ &  
  $K11n67$ \\   
  $K11n68$ &  
  $K11n71$ &  
  $K11n73$ &  
  $K11n74$ &  
  $K11n76$ &  
  $K11n78$ &  
  $K11n79$ &  
  $K11n80$ &  
  $K11n82$ \\   
  $K11n83$ &  
  $K11n84$ &  
  $K11n85$ &  
  $K11n86$ &  
  $K11n87$ &  
  $K11n90$ &  
  $K11n91$ &  
  $K11n92$ &  
  $K11n94$ \\   
  $K11n96$ &  
  $K11n97$ &  
  $K11n98$ &  
  $K11n99$ &  
  $K11n100$ &  
  $K11n101$ &  
  $K11n102$ &  
  $K11n106$ &  
  $K11n110$ \\   
  $K11n111$ &  
  $K11n112$ &  
  $K11n113$ &  
  $K11n114$ &  
  $K11n115$ &  
  $K11n116$ &  
  $K11n117$ &  
  $K11n119$ &  
  $K11n120$ \\   
  $K11n122$ &  
  $K11n123$ &  
  $K11n124$ &  
  $K11n125$ &  
  $K11n127$ &  
  $K11n128$ &  
  $K11n129$ &  
  $K11n130$ &  
  $K11n131$ \\   
  $K11n132$ &  
  $K11n134$ &  
  $K11n135$ &  
  $K11n138$ &  
  $K11n140$ &  
  $K11n142$ &  
  $K11n143$ &  
  $K11n145$ &  
  $K11n146$ \\   
  $K11n150$ &  
  $K11n151$ &  
  $K11n152$ &  
  $K11n153$ &  
  $K11n154$ &  
  $K11n155$ &  
  $K11n156$ &  
  $K11n157$ &  
  $K11n159$ \\   
  $K11n160$ &  
  $K11n161$ &  
  $K11n162$ &  
  $K11n165$ &  
  $K11n166$ &  
  $K11n168$ &  
  $K11n170$ &  
  $K11n172$ &  
  $K11n176$ \\   
  $K11n177$ &  
  $K11n178$ &  
  $K11n179$ &  
  $K11n182$ &  
  $K11n184$ &  
  $K12a1$ &  
  $K12a2$ &  
  $K12a3$ &  
  $K12a5$ \\   
  $K12a6$ &  
  $K12a7$ &  
  $K12a8$ &  
  $K12a9$ &  
  $K12a10$ &  
  $K12a12$ &  
  $K12a13$ &  
  $K12a14$ &  
  $K12a16$ \\  
  $K12a19$ &  
  $K12a20$ &  
  $K12a22$ &  
  $K12a23$ &  
  $K12a25$ &  
  $K12a27$ &  
  $K12a28$ &  
  $K12a31$ &  
  $K12a32$ \\   
  $K12a33$ &  
  $K12a38$ &  
  $K12a39$ &  
  $K12a40$ &  
  $K12a42$ &  
  $K12a46$ &  
  $K12a51$ &  
  $K12a54$ &  
  $K12a57$ \\  
  $K12a58$ &  
  $K12a61$ &  
  $K12a62$ &  
  $K12a63$ &  
  $K12a68$ &  
  $K12a69$ &  
  $K12a71$ &  
  $K12a73$ &  
  $K12a74$ \\  
  $K12a76$ &  
  $K12a77$ &  
  $K12a78$ &  
  $K12a79$ &  
  $K12a80$ &  
  $K12a81$ &  
  $K12a84$ &  
  $K12a86$ &  
  $K12a87$ \\   
  $K12a89$ &  
  $K12a90$ &  
  $K12a91$ &  
  $K12a92$ &  
  $K12a95$ &  
  $K12a98$ &  
  $K12a99$ &  
  $K12a100$ &  
  $K12a103$ \\  
  $K12a104$ &  
  $K12a109$ &  
  $K12a115$ &  
  $K12a118$ &  
  $K12a120$ &  
  $K12a121$ &  
  $K12a122$ &  
  $K12a123$ &  
  $K12a124$ \\   
  $K12a125$ &  
  $K12a126$ &  
  $K12a127$ &  
  $K12a128$ &  
  $K12a129$ &  
  $K12a130$ &  
  $K12a131$ &  
  $K12a132$ &  
  $K12a133$ \\   
  $K12a134$ &  
  $K12a135$ &  
  $K12a136$ &  
  $K12a137$ &  
  $K12a138$ &  
  $K12a139$ &  
  $K12a141$ &  
  $K12a149$ &  
  $K12a150$ \\   
  $K12a151$ &  
  $K12a160$ &  
  $K12a161$ &  
  $K12a163$ &  
  $K12a166$ &  
  $K12a167$ &  
  $K12a168$ &  
  $K12a169$ &  
  $K12a170$ \\   
  $K12a171$ &  
  $K12a172$ &  
  $K12a173$ &  
  $K12a174$ &  
  $K12a177$ &  
  $K12a179$ &  
  $K12a180$ &  
  $K12a181$ &  
  $K12a182$ \\   
  $K12a183$ &  
  $K12a184$ &  
  $K12a185$ &  
  $K12a188$ &  
  $K12a189$ &  
  $K12a190$ &  
  $K12a191$ &  
  $K12a192$ &  
  $K12a196$ \\   
  $K12a197$ &  
  $K12a198$ &  
  $K12a200$ &  
  $K12a202$ &  
  $K12a207$ &  
  $K12a208$ &  
  $K12a209$ &  
  $K12a211$ &  
  $K12a213$ \\   
  $K12a214$ &  
  $K12a217$ &  
  $K12a218$ &  
  $K12a220$ &  
  $K12a221$ &  
  $K12a222$ &  
  $K12a223$ &  
  $K12a224$ &  
  $K12a229$ \\   
  $K12a233$ &  
  $K12a234$ &  
  $K12a235$ &  
  $K12a237$ &  
  $K12a239$ &  
  $K12a240$ &  
  $K12a241$ &  
  $K12a242$ &  
  $K12a243$ \\   
  $K12a244$ &  
  $K12a245$ &  
  $K12a248$ &  
  $K12a249$ &  
  $K12a251$ &  
  $K12a255$ &  
  $K12a256$ &  
  $K12a257$ &  
  $K12a258$ \\   
  $K12a261$ &  
  $K12a262$ &  
  $K12a263$ &  
  $K12a267$ &  
  $K12a270$ &  
  $K12a271$ &  
  $K12a272$ &  
  $K12a273$ &  
  $K12a275$ \\   
  $K12a279$ &  
  $K12a280$ &  
  $K12a282$ &  
  $K12a283$ &  
  $K12a284$ &  
  $K12a286$ &  
  $K12a287$ &  
  $K12a288$ &  
  $K12a290$ \\   
  $K12a291$ &  
  $K12a294$ &  
  $K12a296$ &  
  $K12a297$ &  
  $K12a300$ &  
  $K12a301$ &  
  $K12a302$ &  
  $K12a303$ &  
  $K12a306$ \\  
  $K12a307$ &  
  $K12a308$ &  
  $K12a309$ &  
  $K12a310$ &  
  $K12a312$ &  
  $K12a313$ &  
  $K12a316$ &  
  $K12a318$ &  
  $K12a326$ \\   
  $K12a328$ &  
  $K12a329$ &  
  $K12a330$ &  
  $K12a332$ &  
  $K12a333$ &  
  $K12a334$ &  
  $K12a336$ &  
  $K12a338$ &  
  $K12a339$ \\ 
  $K12a340$ &  
  $K12a341$ &  
  $K12a343$ &  
  $K12a346$ &  
  $K12a347$ &  
  $K12a348$ &  
  $K12a350$ &  
  $K12a351$ &  
  $K12a353$ \\   
  $K12a354$ &  
  $K12a357$ &  
  $K12a358$ &  
  $K12a359$ &  
  $K12a360$ &  
  $K12a365$ &  
  $K12a366$ &  
  $K12a369$ &  
  $K12a370$ \\   
  $K12a371$ &  
  $K12a372$ &  
  $K12a374$ &  
  $K12a376$ &  
  $K12a377$ &  
  $K12a379$ &  
  $K12a383$ &  
  $K12a384$ &  
  $K12a385$ \\   
  $K12a388$ &  
  $K12a390$ &  
  $K12a394$ &  
  $K12a395$ &  
  $K12a399$ &  
  $K12a400$ &  
  $K12a401$ &  
  $K12a402$ &  
  $K12a403$ \\   
  $K12a404$ &  
  $K12a405$ &  
  $K12a406$ &  
  $K12a407$ &  
  $K12a414$ &  
  $K12a415$ &  
  $K12a417$ &  
  $K12a418$ &  
  $K12a419$ \\   
  $K12a423$ &  
  $K12a424$ &  
  $K12a425$ &  
  $K12a426$ &  
  $K12a427$ &  
  $K12a429$ &  
  $K12a430$ &  
  $K12a435$ &  
  $K12a436$ \\   
  $K12a437$ &  
  $K12a438$ &  
  $K12a439$ &  
  $K12a441$ &  
  $K12a446$ &  
  $K12a447$ &  
  $K12a448$ &  
  $K12a450$ &  
  $K12a451$ \\   
  $K12a452$ &  
  $K12a453$ &  
  $K12a454$ &  
  $K12a455$ &  
  $K12a456$ &  
  $K12a457$ &  
  $K12a458$ &  
  $K12a459$ &  
  $K12a462$ \\   
  $K12a463$ &  
  $K12a464$ &  
  $K12a467$ &  
  $K12a468$ &  
  $K12a469$ &  
  $K12a470$ &  
  $K12a471$ &  
  $K12a473$ &  
  $K12a474$ \\   
  $K12a475$ &  
  $K12a476$ &  
  $K12a477$ &  
  $K12a478$ &  
  $K12a479$ &  
  $K12a480$ &  
  $K12a482$ &  
  $K12a483$ &  
  $K12a484$ \\   
  $K12a485$ &  
  $K12a486$ &  
  $K12a487$ &  
  $K12a488$ &  
  $K12a489$ &  
  $K12a491$ &  
  $K12a492$ &  
  $K12a494$ &  
  $K12a496$ \\   
  $K12a497$ &  
  $K12a498$ &  
  $K12a500$ &  
  $K12a505$ &  
  $K12a506$ &  
  $K12a509$ &  
  $K12a510$ &  
  $K12a512$ &  
  $K12a513$ \\  
  $K12a514$ &  
  $K12a515$ &  
  $K12a516$ &  
  $K12a518$ &  
  $K12a524$ &  
  $K12a525$ &  
  $K12a526$ &  
  $K12a528$ &  
  $K12a529$ \\   
  $K12a530$ &  
  $K12a531$ &  
  $K12a533$ &  
  $K12a535$ &  
  $K12a537$ &  
  $K12a542$ &  
  $K12a543$ &  
  $K12a544$ &  
  $K12a545$ \\   
  $K12a546$ &  
  $K12a549$ &  
  $K12a551$ &  
  $K12a552$ &  
  $K12a555$ &  
  $K12a556$ &  
  $K12a557$ &  
  $K12a559$ &  
  $K12a560$ \\   
  $K12a562$ &  
  $K12a564$ &  
  $K12a566$ &  
  $K12a567$ &  
  $K12a568$ &  
  $K12a571$ &  
  $K12a572$ &  
  $K12a573$ &  
  $K12a579$ \\   
  $K12a581$ &  
  $K12a582$ &  
  $K12a583$ &  
  $K12a585$ &  
  $K12a588$ &  
  $K12a589$ &  
  $K12a592$ &  
  $K12a593$ &  
  $K12a594$ \\   
  $K12a596$ &  
  $K12a598$ &  
  $K12a599$ &  
  $K12a601$ &  
  $K12a603$ &  
  $K12a604$ &  
  $K12a606$ &  
  $K12a607$ &  
  $K12a608$ \\   
  $K12a609$ &  
  $K12a613$ &  
  $K12a614$ &  
  $K12a617$ &  
  $K12a618$ &  
  $K12a619$ &  
  $K12a621$ &  
  $K12a622$ &  
  $K12a624$ \\   
  $K12a625$ &  
  $K12a626$ &  
  $K12a627$ &  
  $K12a628$ &  
  $K12a629$ &  
  $K12a630$ &  
  $K12a631$ &  
  $K12a632$ &  
  $K12a633$ \\   
  $K12a634$ &  
  $K12a635$ &  
  $K12a637$ &  
  $K12a638$ &  
  $K12a639$ &  
  $K12a640$ &  
  $K12a641$ &  
  $K12a642$ &  
  $K12a644$ \\   
  $K12a645$ &  
  $K12a646$ &  
  $K12a650$ &  
  $K12a652$ &  
  $K12a655$ &  
  $K12a656$ &  
  $K12a657$ &  
  $K12a658$ &  
  $K12a662$ \\   
  $K12a663$ &  
  $K12a665$ &  
  $K12a666$ &  
  $K12a668$ &  
  $K12a670$ &  
  $K12a671$ &  
  $K12a672$ &  
  $K12a673$ &  
  $K12a674$ \\   
  $K12a676$ &  
  $K12a677$ &  
  $K12a678$ &  
  $K12a680$ &  
  $K12a682$ &  
  $K12a687$ &  
  $K12a688$ &  
  $K12a689$ &  
  $K12a690$ \\   
  $K12a691$ &  
  $K12a692$ &  
  $K12a695$ &  
  $K12a696$ &  
  $K12a697$ &  
  $K12a698$ &  
  $K12a699$ &  
  $K12a700$ &  
  $K12a702$ \\   
  $K12a704$ &  
  $K12a705$ &  
  $K12a707$ &  
  $K12a708$ &  
  $K12a709$ &  
  $K12a710$ &  
  $K12a715$ &  
  $K12a719$ &  
  $K12a721$ \\   
  $K12a724$ &  
  $K12a729$ &  
  $K12a730$ &  
  $K12a732$ &  
  $K12a734$ &  
  $K12a735$ &  
  $K12a736$ &  
  $K12a738$ &  
  $K12a741$ \\   
  $K12a744$ &  
  $K12a745$ &  
  $K12a746$ &  
  $K12a747$ &  
  $K12a748$ &  
  $K12a749$ &  
  $K12a751$ &  
  $K12a752$ &  
  $K12a753$ \\   
  $K12a754$ &  
  $K12a756$ &  
  $K12a758$ &  
  $K12a760$ &  
  $K12a761$ &  
  $K12a765$ &  
  $K12a766$ &  
  $K12a767$ &  
  $K12a770$ \\   
  $K12a771$ &  
  $K12a772$ &  
  $K12a773$ &  
  $K12a774$ &  
  $K12a775$ &  
  $K12a776$ &  
  $K12a778$ &  
  $K12a779$ &  
  $K12a781$ \\   
  $K12a782$ &  
  $K12a783$ &  
  $K12a784$ &  
  $K12a786$ &  
  $K12a789$ &  
  $K12a792$ &  
  $K12a793$ &  
  $K12a797$ &  
  $K12a799$ \\   
  $K12a802$ &  
  $K12a806$ &  
  $K12a807$ &  
  $K12a808$ &  
  $K12a815$ &  
  $K12a819$ &  
  $K12a820$ &  
  $K12a821$ &  
  $K12a824$ \\   
  $K12a825$ &  
  $K12a826$ &  
  $K12a831$ &  
  $K12a833$ &  
  $K12a834$ &  
  $K12a835$ &  
  $K12a836$ &  
  $K12a837$ &  
  $K12a839$ \\   
  $K12a841$ &  
  $K12a844$ &  
  $K12a849$ &  
  $K12a851$ &  
  $K12a853$ &  
  $K12a854$ &  
  $K12a856$ &  
  $K12a857$ &  
  $K12a858$ \\   
  $K12a859$ &  
  $K12a860$ &  
  $K12a862$ &  
  $K12a863$ &  
  $K12a864$ &  
  $K12a865$ &  
  $K12a866$ &  
  $K12a867$ &  
  $K12a869$ \\   
  $K12a870$ &  
  $K12a871$ &  
  $K12a872$ &  
  $K12a875$ &  
  $K12a878$ &  
  $K12a879$ &  
  $K12a883$ &  
  $K12a885$ &  
  $K12a887$ \\   
  $K12a888$ &  
  $K12a890$ &  
  $K12a892$ &  
  $K12a893$ &  
  $K12a894$ &  
  $K12a897$ &  
  $K12a898$ &  
  $K12a901$ &  
  $K12a903$ \\   
  $K12a904$ &  
  $K12a908$ &  
  $K12a910$ &  
  $K12a911$ &  
  $K12a912$ &  
  $K12a913$ &  
  $K12a914$ &  
  $K12a915$ &  
  $K12a916$ \\   
  $K12a919$ &  
  $K12a920$ &  
  $K12a921$ &  
  $K12a924$ &  
  $K12a926$ &  
  $K12a927$ &  
  $K12a928$ &  
  $K12a929$ &  
  $K12a930$ \\   
  $K12a931$ &  
  $K12a933$ &  
  $K12a934$ &  
  $K12a935$ &  
  $K12a936$ &  
  $K12a939$ &  
  $K12a940$ &  
  $K12a941$ &  
  $K12a942$ \\   
  $K12a944$ &  
  $K12a945$ &  
  $K12a947$ &  
  $K12a948$ &  
  $K12a949$ &  
  $K12a950$ &  
  $K12a951$ &  
  $K12a954$ &  
  $K12a959$ \\  
  $K12a960$ &  
  $K12a961$ &  
  $K12a963$ &  
  $K12a964$ &  
  $K12a965$ &  
  $K12a966$ &  
  $K12a968$ &  
  $K12a969$ &  
  $K12a975$ \\   
  $K12a976$ &  
  $K12a979$ &  
  $K12a981$ &  
  $K12a982$ &  
  $K12a984$ &  
  $K12a985$ &  
  $K12a988$ &  
  $K12a989$ &  
  $K12a990$ \\  
  $K12a991$ &  
  $K12a992$ &  
  $K12a993$ &  
  $K12a994$ &  
  $K12a997$ &  
  $K12a998$ &  
  $K12a999$ &   
  $K12a1000$ &  
  $K12a1001$ \\   
  $K12a1002$ &  
  $K12a1003$ &  
  $K12a1005$ &  
  $K12a1006$ &  
  $K12a1008$ &  
  $K12a1009$ &  
  $K12a1010$ &  
  $K12a1011$ &  
  $K12a1013$ \\   
  $K12a1014$ &  
  $K12a1015$ &  
  $K12a1016$ &  
  $K12a1018$ &  
  $K12a1019$ &  
  $K12a1023$ &  
  $K12a1026$ &  
  $K12a1027$ &  
  $K12a1028$ \\   
  $K12a1029$ &  
  $K12a1031$ &  
  $K12a1032$ &  
  $K12a1034$ &  
  $K12a1039$ &  
  $K12a1046$ &  
  $K12a1047$ &  
  $K12a1048$ &  
  $K12a1051$ \\   
  $K12a1052$ &  
  $K12a1053$ &  
  $K12a1055$ &  
  $K12a1056$ &  
  $K12a1057$ &  
  $K12a1058$ &  
  $K12a1059$ &  
  $K12a1060$ &  
  $K12a1061$ \\   
  $K12a1062$ &  
  $K12a1063$ &  
  $K12a1064$ &  
  $K12a1065$ &  
  $K12a1066$ &  
  $K12a1067$ &  
  $K12a1068$ &  
  $K12a1069$ &  
  $K12a1070$ \\   
  $K12a1071$ &  
  $K12a1073$ &  
  $K12a1076$ &  
  $K12a1077$ &  
  $K12a1078$ &  
  $K12a1080$ &  
  $K12a1081$ &  
  $K12a1082$ &  
  $K12a1083$ \\   
  $K12a1086$ &  
  $K12a1087$ &  
  $K12a1090$ &  
  $K12a1093$ &  
  $K12a1094$ &  
  $K12a1096$ &  
  $K12a1098$ &  
  $K12a1099$ &  
  $K12a1100$ \\   
  $K12a1102$ &  
  $K12a1103$ &  
  $K12a1104$ &  
  $K12a1105$ &  
  $K12a1106$ &  
  $K12a1108$ &  
  $K12a1109$ &  
  $K12a1110$ &  
  $K12a1115$ \\   
  $K12a1116$ &  
  $K12a1118$ &  
  $K12a1119$ &  
  $K12a1120$ &  
  $K12a1121$ &  
  $K12a1124$ &  
  $K12a1126$ &  
  $K12a1127$ &  
  $K12a1128$ \\   
  $K12a1129$ &  
  $K12a1130$ &  
  $K12a1132$ &  
  $K12a1135$ &  
  $K12a1137$ &  
  $K12a1139$ &  
  $K12a1140$ &  
  $K12a1143$ &  
  $K12a1144$ \\   
  $K12a1145$ &  
  $K12a1146$ &  
  $K12a1147$ &  
  $K12a1148$ &  
  $K12a1151$ &  
  $K12a1153$ &  
  $K12a1154$ &  
  $K12a1158$ &  
  $K12a1160$ \\   
  $K12a1161$ &  
  $K12a1163$ &  
  $K12a1165$ &  
  $K12a1168$ &  
  $K12a1169$ &  
  $K12a1170$ &  
  $K12a1172$ &  
  $K12a1173$ &  
  $K12a1174$ \\   
  $K12a1175$ &  
  $K12a1176$ &  
  $K12a1177$ &  
  $K12a1178$ &  
  $K12a1180$ &  
  $K12a1181$ &  
  $K12a1182$ &  
  $K12a1185$ &  
  $K12a1186$ \\   
  $K12a1187$ &  
  $K12a1188$ &  
  $K12a1189$ &  
  $K12a1192$ &  
  $K12a1195$ &  
  $K12a1196$ &  
  $K12a1197$ &  
  $K12a1198$ &  
  $K12a1199$ \\   
  $K12a1202$ &  
  $K12a1203$ &  
  $K12a1204$ &  
  $K12a1211$ &  
  $K12a1212$ &  
  $K12a1213$ &  
  $K12a1215$ &  
  $K12a1217$ &  
  $K12a1218$ \\   
  $K12a1219$ &  
  $K12a1223$ &  
  $K12a1224$ &  
  $K12a1225$ &  
  $K12a1227$ &  
  $K12a1228$ &  
  $K12a1230$ &  
  $K12a1231$ &  
  $K12a1234$ \\   
  $K12a1235$ &  
  $K12a1236$ &  
  $K12a1237$ &  
  $K12a1238$ &  
  $K12a1239$ &  
  $K12a1241$ &  
  $K12a1244$ &  
  $K12a1245$ &  
  $K12a1246$ \\   
  $K12a1248$ &  
  $K12a1249$ &  
  $K12a1250$ &  
  $K12a1251$ &  
  $K12a1252$ &  
  $K12a1253$ &  
  $K12a1254$ &  
  $K12a1255$ &  
  $K12a1257$ \\   
  $K12a1258$ &  
  $K12a1259$ &  
  $K12a1260$ &  
  $K12a1261$ &  
  $K12a1262$ &  
  $K12a1263$ &  
  $K12a1265$ &  
  $K12a1266$ &  
  $K12a1267$ \\   
  $K12a1268$ &  
  $K12a1269$ &  
  $K12a1270$ &  
  $K12a1271$ &  
  $K12a1272$ &  
  $K12a1273$ &  
  $K12a1275$ &  
  $K12a1276$ &  
  $K12a1277$ \\   
  $K12a1281$ &  
  $K12a1282$ &  
  $K12a1283$ &  
  $K12a1285$ &  
  $K12n1$ &  
  $K12n2$ &  
  $K12n3$ &  
  $K12n4$ &  
  $K12n5$ \\   
  $K12n7$ &  
  $K12n9$ &  
  $K12n10$ &  
  $K12n11$ &  
  $K12n12$ &  
  $K12n13$ &  
  $K12n14$ &  
  $K12n15$ &  
  $K12n16$ \\   
  $K12n17$ &  
  $K12n18$ &  
  $K12n19$ &  
  $K12n20$ &  
  $K12n21$ &  
  $K12n22$ &  
  $K12n23$ &  
  $K12n24$ &  
  $K12n25$ \\   
  $K12n26$ &  
  $K12n27$ &  
  $K12n28$ &  
  $K12n29$ &  
  $K12n31$ &  
  $K12n32$ &  
  $K12n33$ &  
  $K12n34$ &  
  $K12n35$ \\   
  $K12n36$ &  
  $K12n37$ &  
  $K12n38$ &  
  $K12n39$ &  
  $K12n40$ &  
  $K12n41$ &  
  $K12n42$ &  
  $K12n43$ &  
  $K12n44$ \\   
  $K12n45$ &  
  $K12n46$ &  
  $K12n47$ &  
  $K12n48$ &  
  $K12n49$ &  
  $K12n50$ &  
  $K12n51$ &  
  $K12n53$ &  
  $K12n55$ \\   
  $K12n56$ &  
  $K12n57$ &  
  $K12n60$ &  
  $K12n62$ &  
  $K12n63$ &  
  $K12n65$ &  
  $K12n66$ &  
  $K12n69$ &  
  $K12n70$ \\   
  $K12n71$ &  
  $K12n73$ &  
  $K12n75$ &  
  $K12n76$ &  
  $K12n78$ &  
  $K12n79$ &  
  $K12n80$ &  
  $K12n81$ &  
  $K12n82$ \\   
  $K12n84$ &  
  $K12n85$ &  
  $K12n86$ &  
  $K12n87$ &  
  $K12n89$ &  
  $K12n90$ &  
  $K12n94$ &  
  $K12n95$ &  
  $K12n97$ \\   
  $K12n98$ &  
  $K12n99$ &  
  $K12n101$ &  
  $K12n102$ &  
  $K12n104$ &  
  $K12n106$ &  
  $K12n111$ &  
  $K12n112$ &  
  $K12n113$ \\   
  $K12n114$ &  
  $K12n115$ &  
  $K12n116$ &  
  $K12n118$ &  
  $K12n120$ &  
  $K12n121$ &  
  $K12n122$ &  
  $K12n124$ &  
  $K12n125$ \\   
  $K12n126$ &  
  $K12n127$ &  
  $K12n128$ &  
  $K12n129$ &  
  $K12n130$ &  
  $K12n131$ &  
  $K12n132$ &  
  $K12n134$ &  
  $K12n140$ \\   
  $K12n141$ &  
  $K12n142$ &  
  $K12n144$ &  
  $K12n145$ &  
  $K12n146$ &  
  $K12n151$ &  
  $K12n152$ &  
  $K12n154$ &  
  $K12n157$ \\   
  $K12n158$ &  
  $K12n159$ &  
  $K12n160$ &  
  $K12n161$ &  
  $K12n162$ &  
  $K12n164$ &  
  $K12n165$ &  
  $K12n170$ &  
  $K12n171$ \\   
  $K12n173$ &  
  $K12n174$ &  
  $K12n176$ &  
  $K12n179$ &  
  $K12n180$ &  
  $K12n181$ &  
  $K12n182$ &  
  $K12n184$ &  
  $K12n186$ \\   
  $K12n188$ &  
  $K12n189$ &  
  $K12n190$ &  
  $K12n192$ &  
  $K12n193$ &  
  $K12n194$ &  
  $K12n195$ &  
  $K12n197$ &  
  $K12n198$ \\   
  $K12n199$ &  
  $K12n200$ &  
  $K12n202$ &  
  $K12n206$ &  
  $K12n210$ &  
  $K12n211$ &  
  $K12n212$ &  
  $K12n213$ &  
  $K12n214$ \\  
  $K12n215$ &  
  $K12n216$ &  
  $K12n218$ &  
  $K12n221$ &  
  $K12n223$ &  
  $K12n224$ &  
  $K12n225$ &  
  $K12n230$ &  
  $K12n231$ \\   
  $K12n232$ &  
  $K12n236$ &  
  $K12n237$ &  
  $K12n238$ &  
  $K12n239$ &  
  $K12n241$ &  
  $K12n246$ &  
  $K12n247$ &  
  $K12n248$ \\   
  $K12n249$ &  
  $K12n250$ &  
  $K12n252$ &  
  $K12n253$ &  
  $K12n255$ &  
  $K12n256$ &  
  $K12n257$ &  
  $K12n258$ &  
  $K12n260$ \\   
  $K12n262$ &  
  $K12n263$ &  
  $K12n264$ &  
  $K12n265$ &  
  $K12n266$ &  
  $K12n267$ &  
  $K12n268$ &  
  $K12n270$ &  
  $K12n271$ \\   
  $K12n272$ &  
  $K12n274$ &  
  $K12n275$ &  
  $K12n277$ &  
  $K12n278$ &  
  $K12n279$ &  
  $K12n280$ &  
  $K12n281$ &  
  $K12n282$ \\   
  $K12n283$ &  
  $K12n284$ &  
  $K12n285$ &  
  $K12n286$ &  
  $K12n287$ &  
  $K12n288$ &  
  $K12n295$ &  
  $K12n297$ &  
  $K12n298$ \\   
  $K12n300$ &  
  $K12n301$ &  
  $K12n302$ &  
  $K12n304$ &  
  $K12n306$ &  
  $K12n307$ &  
  $K12n309$ &  
  $K12n310$ &  
  $K12n311$ \\   
  $K12n312$ &  
  $K12n313$ &  
  $K12n314$ &  
  $K12n315$ &  
  $K12n317$ &  
  $K12n318$ &  
  $K12n319$ &  
  $K12n320$ &  
  $K12n322$ \\   
  $K12n323$ &  
  $K12n324$ &  
  $K12n325$ &  
  $K12n330$ &  
  $K12n333$ &  
  $K12n335$ &  
  $K12n336$ &  
  $K12n337$ &  
  $K12n339$ \\   
  $K12n340$ &  
  $K12n342$ &  
  $K12n343$ &  
  $K12n344$ &  
  $K12n345$ &  
  $K12n346$ &  
  $K12n347$ &  
  $K12n348$ &  
  $K12n350$ \\   
  $K12n351$ &  
  $K12n352$ &  
  $K12n353$ &  
  $K12n354$ &  
  $K12n355$ &  
  $K12n356$ &  
  $K12n357$ &  
  $K12n358$ &  
  $K12n359$ \\   
  $K12n360$ &  
  $K12n361$ &  
  $K12n362$ &  
  $K12n363$ &  
  $K12n364$ &  
  $K12n365$ &  
  $K12n367$ &  
  $K12n369$ &  
  $K12n370$ \\   
  $K12n371$ &  
  $K12n372$ &  
  $K12n376$ &  
  $K12n377$ &  
  $K12n378$ &  
  $K12n379$ &  
  $K12n380$ &  
  $K12n381$ &  
  $K12n382$ \\   
  $K12n383$ &  
  $K12n384$ &  
  $K12n385$ &  
  $K12n388$ &  
  $K12n389$ &  
  $K12n390$ &  
  $K12n391$ &  
  $K12n392$ &  
  $K12n393$ \\   
  $K12n394$ &  
  $K12n397$ &  
  $K12n399$ &  
  $K12n400$ &  
  $K12n401$ &  
  $K12n408$ &  
  $K12n409$ &  
  $K12n410$ &  
  $K12n411$ \\   
  $K12n412$ &  
  $K12n414$ &  
  $K12n415$ &  
  $K12n416$ &  
  $K12n420$ &  
  $K12n421$ &  
  $K12n422$ &  
  $K12n423$ &  
  $K12n424$ \\   
  $K12n427$ &  
  $K12n429$ &  
  $K12n430$ &  
  $K12n431$ &  
  $K12n434$ &  
  $K12n435$ &  
  $K12n437$ &  
  $K12n439$ &  
  $K12n440$ \\   
  $K12n442$ &  
  $K12n443$ &  
  $K12n444$ &  
  $K12n446$ &  
  $K12n447$ &  
  $K12n448$ &  
  $K12n449$ &  
  $K12n450$ &  
  $K12n451$ \\   
  $K12n452$ &  
  $K12n454$ &  
  $K12n456$ &  
  $K12n457$ &  
  $K12n458$ &  
  $K12n461$ &  
  $K12n462$ &  
  $K12n463$ &  
  $K12n464$ \\   
  $K12n466$ &  
  $K12n467$ &  
  $K12n468$ &  
  $K12n469$ &  
  $K12n470$ &  
  $K12n471$ &  
  $K12n475$ &  
  $K12n478$ &  
  $K12n479$ \\   
  $K12n481$ &  
  $K12n482$ &  
  $K12n483$ &  
  $K12n484$ &  
  $K12n485$ &  
  $K12n486$ &  
  $K12n487$ &  
  $K12n488$ &  
  $K12n489$ \\   
  $K12n490$ &  
  $K12n491$ &  
  $K12n492$ &  
  $K12n493$ &  
  $K12n497$ &  
  $K12n498$ &  
  $K12n499$ &  
  $K12n500$ &  
  $K12n501$ \\   
  $K12n504$ &  
  $K12n505$ &  
  $K12n506$ &  
  $K12n507$ &  
  $K12n511$ &  
  $K12n512$ &  
  $K12n514$ &  
  $K12n515$ &  
  $K12n517$ \\   
  $K12n519$ &  
  $K12n520$ &  
  $K12n521$ &  
  $K12n522$ &  
  $K12n523$ &  
  $K12n531$ &  
  $K12n533$ &  
  $K12n534$ &  
  $K12n535$ \\   
  $K12n536$ &  
  $K12n539$ &  
  $K12n541$ &  
  $K12n543$ &  
  $K12n544$ &  
  $K12n546$ &  
  $K12n547$ &  
  $K12n548$ &  
  $K12n550$ \\   
  $K12n551$ &  
  $K12n552$ &  
  $K12n557$ &  
  $K12n558$ &  
  $K12n559$ &  
  $K12n560$ &  
  $K12n561$ &  
  $K12n562$ &  
  $K12n563$ \\   
  $K12n564$ &  
  $K12n566$ &  
  $K12n567$ &  
  $K12n568$ &  
  $K12n569$ &  
  $K12n572$ &  
  $K12n573$ &  
  $K12n578$ &  
  $K12n579$ \\   
  $K12n580$ &  
  $K12n584$ &  
  $K12n586$ &  
  $K12n587$ &  
  $K12n588$ &  
  $K12n592$ &  
  $K12n595$ &  
  $K12n596$ &  
  $K12n597$ \\   
  $K12n599$ &  
  $K12n601$ &  
  $K12n605$ &  
  $K12n606$ &  
  $K12n607$ &  
  $K12n608$ &  
  $K12n610$ &  
  $K12n611$ &  
  $K12n612$ \\   
  $K12n613$ &  
  $K12n614$ &  
  $K12n616$ &  
  $K12n617$ &  
  $K12n618$ &  
  $K12n620$ &  
  $K12n621$ &  
  $K12n625$ &  
  $K12n628$ \\   
  $K12n629$ &  
  $K12n630$ &  
  $K12n631$ &  
  $K12n632$ &  
  $K12n633$ &  
  $K12n634$ &  
  $K12n635$ &  
  $K12n636$ &  
  $K12n637$ \\   
  $K12n643$ &  
  $K12n645$ &  
  $K12n646$ &  
  $K12n650$ &  
  $K12n651$ &  
  $K12n656$ &  
  $K12n657$ &  
  $K12n661$ &  
  $K12n662$ \\   
  $K12n663$ &  
  $K12n664$ &  
  $K12n665$ &  
  $K12n667$ &  
  $K12n669$ &  
  $K12n670$ &  
  $K12n672$ &  
  $K12n673$ &  
  $K12n675$ \\   
  $K12n676$ &  
  $K12n677$ &  
  $K12n678$ &  
  $K12n681$ &  
  $K12n683$ &  
  $K12n684$ &  
  $K12n685$ &  
  $K12n686$ &  
  $K12n687$ \\   
  $K12n690$ &  
  $K12n695$ &  
  $K12n699$ &  
  $K12n701$ &  
  $K12n702$ &  
  $K12n704$ &  
  $K12n705$ &  
  $K12n706$ &  
  $K12n708$ \\  
  $K12n709$ &  
  $K12n710$ &  
  $K12n711$ &  
  $K12n712$ &  
  $K12n713$ &  
  $K12n715$ &  
  $K12n716$ &  
  $K12n717$ &  
  $K12n718$ \\   
  $K12n719$ &  
  $K12n726$ &  
  $K12n727$ &  
  $K12n730$ &  
  $K12n731$ &  
  $K12n732$ &  
  $K12n733$ &  
  $K12n735$ &  
  $K12n736$ \\   
  $K12n740$ &  
  $K12n741$ &  
  $K12n742$ &  
  $K12n743$ &  
  $K12n746$ &  
  $K12n748$ &  
  $K12n749$ &  
  $K12n751$ &  
  $K12n752$ \\   
  $K12n753$ &  
  $K12n754$ &  
  $K12n755$ &  
  $K12n757$ &  
  $K12n759$ &  
  $K12n760$ &  
  $K12n763$ &  
  $K12n766$ &  
  $K12n767$ \\   
  $K12n768$ &  
  $K12n769$ &  
  $K12n771$ &  
  $K12n772$ &  
  $K12n773$ &  
  $K12n774$ &  
  $K12n775$ &  
  $K12n777$ &  
  $K12n779$ \\   
  $K12n780$ &  
  $K12n781$ &  
  $K12n782$ &  
  $K12n785$ &  
  $K12n786$ &  
  $K12n787$ &  
  $K12n788$ &  
  $K12n789$ &  
  $K12n790$ \\   
  $K12n792$ &  
  $K12n794$ &  
  $K12n795$ &  
  $K12n797$ &  
  $K12n798$ &  
  $K12n800$ &  
  $K12n801$ &  
  $K12n802$ &  
  $K12n804$ \\   
  $K12n805$ &  
  $K12n808$ &  
  $K12n809$ &  
  $K12n810$ &  
  $K12n811$ &  
  $K12n812$ &  
  $K12n814$ &  
  $K12n815$ &  
  $K12n816$ \\   
  $K12n817$ &  
  $K12n819$ &  
  $K12n821$ &  
  $K12n822$ &  
  $K12n824$ &  
  $K12n825$ &  
  $K12n826$ &  
  $K12n827$ &  
  $K12n828$ \\   
  $K12n829$ &  
  $K12n831$ &  
  $K12n833$ &  
  $K12n835$ &  
  $K12n838$ &  
  $K12n840$ &  
  $K12n841$ &  
  $K12n844$ &  
  $K12n846$ \\   
  $K12n847$ &  
  $K12n852$ &  
  $K12n853$ &  
  $K12n854$ &  
  $K12n855$ &  
  $K12n857$ &  
  $K12n858$ &  
  $K12n860$ &  
  $K12n861$ \\   
  $K12n862$ &  
  $K12n864$ &  
  $K12n865$ &  
  $K12n868$ &  
  $K12n870$ &  
  $K12n871$ &  
  $K12n873$ &  
  $K12n876$ &  
  $K12n878$ \\   
  $K12n879$ &  
  $K12n883$ &  
  $K12n884$ &  
  $K12n886$ &&&&&\\ 
\end{longtable}}

{\fontsize{6}{9}\selectfont
\begin{longtable}{@{}lllllllll@{}}
\caption{Census knots that admit a friend.} \label{tab:census} \\

\endfirsthead

\multicolumn{9}{c}%
{{\tablename\ \thetable{} -- continued from previous page}} \\
\endhead

\hline \multicolumn{9}{r}{{Continued on next page}} \\ 
\endfoot

\hline \hline
\endlastfoot

  $K5\_8$ &
  $K5\_13$ &
  $K5\_19$ &
  $K6\_9$ &
  $K6\_10$ &
  $K6\_37$ &
  $K6\_40$ &
  $K7\_11$ &
  $K7\_18$ \\
  $K7\_24$ &
  $K7\_44$ &
  $K7\_49$ &
  $K7\_92$ &
  $K7\_105$ &
  $K7\_114$ &
  $K7\_117$ &
  $K7\_121$ &
  $K7\_125$ \\
  $K7\_127$ &
  $K7\_128$ &
  $K7\_129$ &
  $K8\_10$ &
  $K8\_20$ &
  $K8\_44$ &
  $K8\_48$ &
  $K8\_51$ &
  $K8\_137$ \\
  $K8\_141$ &
  $K8\_169$ &
  $K8\_170$ &
  $K8\_179$ &
  $K8\_183$ &
  $K8\_194$ &
  $K8\_198$ &
  $K8\_199$ &
  $K8\_205$ \\
  $K8\_206$ &
  $K8\_208$ &
  $K8\_221$ &
  $K8\_224$ &
  $K8\_226$ &
  $K8\_234$ &
  $K8\_236$ &
  $K8\_238$ &
  $K8\_240$ \\
  $K8\_241$ &
  $K8\_245$ &
  $K8\_247$ &
  $K8\_249$ &
  $K8\_250$ &
  $K8\_252$ &
  $K8\_257$ &
  $K8\_259$ &
  $K8\_263$ \\
  $K8\_265$ &
  $K8\_270$ &
  $K8\_272$ &
  $K8\_277$ &
  $K8\_280$ &
  $K8\_281$ &
  $K8\_286$ &
  $K8\_287$ &
  $K8\_288$ \\
  $K8\_291$ &
  $K8\_292$ &
  $K8\_293$ &
  $K8\_294$ &
  $K8\_295$ &
  $K8\_297$ &
  $K8\_301$ &
  $K9\_9$ &
  $K9\_21$ \\
  $K9\_59$ &
  $K9\_74$ &
  $K9\_149$ &
  $K9\_256$ &
  $K9\_259$ &
  $K9\_261$ &
  $K9\_274$ &
  $K9\_281$ &
  $K9\_314$ \\
  $K9\_326$ &
  $K9\_361$ &
  $K9\_383$ &
  $K9\_386$ &
  $K9\_387$ &
  $K9\_395$ &
  $K9\_404$ &
  $K9\_415$ &
  $K9\_426$ \\
  $K9\_433$ &
  $K9\_434$ &
  $K9\_437$ &
  $K9\_439$ &
  $K9\_444$ &
  $K9\_447$ &
  $K9\_456$ &
  $K9\_465$ &
  $K9\_468$ \\
  $K9\_488$ &
  $K9\_500$ &
  $K9\_505$ &
  $K9\_509$ &
  $K9\_510$ &
  $K9\_514$ &
  $K9\_515$ &
  $K9\_517$ &
  $K9\_532$ \\
  $K9\_533$ &
  $K9\_534$ &
  $K9\_535$ &
  $K9\_536$ &
  $K9\_537$ &
  $K9\_563$ &
  $K9\_564$ &
  $K9\_574$ &
  $K9\_580$ \\
  $K9\_585$ &
  $K9\_586$ &
  $K9\_591$ &
  $K9\_592$ &
  $K9\_596$ &
  $K9\_597$ &
  $K9\_608$ &
  $K9\_610$ &
  $K9\_612$ \\
  $K9\_622$ &
  $K9\_626$ &
  $K9\_630$ &
  $K9\_633$ &
  $K9\_634$ &
  $K9\_639$ &
  $K9\_643$ &
  $K9\_645$ &
  $K9\_648$ \\
  $K9\_650$ &
  $K9\_651$ &
  $K9\_655$ &
  $K9\_663$ &
  $K9\_664$ &
  $K9\_668$ &
  $K9\_670$ &
  $K9\_672$ &
  $K9\_675$ \\
  $K9\_677$ &
  $K9\_678$ &
  $K9\_681$ &
  $K9\_683$ &
  $K9\_684$ &
  $K9\_686$ &
  $K9\_687$ &
  $K9\_689$ &
  $K9\_692$ \\
  $K9\_694$ &
  $K9\_696$ &
  $K9\_700$ &
  $K9\_701$ &
  $K9\_702$ &
  $K9\_703$ &
  $K9\_704$ &
  $K9\_707$ &
  $K9\_709$ \\
  $K9\_712$ &
  $K9\_717$ &
  $K9\_719$ &
  $K9\_720$ &
  $K9\_722$ &
  $K9\_724$ &
  $K9\_725$ &
  $K9\_727$ &
  $K9\_736$ \\
  $K9\_738$ &
  $K9\_740$ &
  $K9\_741$ &
  $K9\_743$ &
  $K9\_744$ &
  $K9\_746$ &
  $K9\_754$ &
  $K9\_757$ &
  $K9\_758$ \\
  $K9\_759$ &
  $K9\_760$ &
  $K9\_762$ &&&&&&\\
\end{longtable}}

\newpage
 \let\MRhref\undefined
 \bibliographystyle{hamsalpha}
 \bibliography{friends}

\end{document}

%% file: RGB.pdf_tex
\begingroup%
  \makeatletter%
  \providecommand\color[2][]{%
    \errmessage{(Inkscape) Color is used for the text in Inkscape, but the package 'color.sty' is not loaded}%
    \renewcommand\color[2][]{}%
  }%
  \providecommand\transparent[1]{%
    \errmessage{(Inkscape) Transparency is used (non-zero) for the text in Inkscape, but the package 'transparent.sty' is not loaded}%
    \renewcommand\transparent[1]{}%
  }%
  \providecommand\rotatebox[2]{#2}%
  \newcommand*\fsize{\dimexpr\f@size pt\relax}%
  \newcommand*\lineheight[1]{\fontsize{\fsize}{#1\fsize}\selectfont}%
  \ifx\svgwidth\undefined%
    \setlength{\unitlength}{278.5409643bp}%
    \ifx\svgscale\undefined%
      \relax%
    \else%
      \setlength{\unitlength}{\unitlength * \real{\svgscale}}%
    \fi%
  \else%
    \setlength{\unitlength}{\svgwidth}%
  \fi%
  \global\let\svgwidth\undefined%
  \global\let\svgscale\undefined%
  \makeatother%
  \begin{picture}(1,0.22415798)%
    \lineheight{1}%
    \setlength\tabcolsep{0pt}%
    \put(0,0){\includegraphics[width=\unitlength,page=1]{RGB.pdf}}%
    \put(0.43770549,0.10566208){\color[rgb]{0,0,0.01176471}\makebox(0,0)[lt]{\lineheight{1.25}\smash{\begin{tabular}[t]{l}$\cong$\end{tabular}}}}%
    \put(0.7812902,0.10124184){\color[rgb]{0,0,0.01176471}\makebox(0,0)[lt]{\lineheight{1.25}\smash{\begin{tabular}[t]{l}$\cong$\end{tabular}}}}%
    \put(0.0947974,0.10364319){\color[rgb]{0,0,0.01176471}\makebox(0,0)[lt]{\lineheight{1.25}\smash{\begin{tabular}[t]{l}$\cong$\end{tabular}}}}%
    \put(0.00014212,0.10346021){\color[rgb]{0,0,0.01176471}\makebox(0,0)[lt]{\lineheight{1.25}\smash{\begin{tabular}[t]{l}$K(0)$\end{tabular}}}}%
    \put(0.85220339,0.10141082){\color[rgb]{0,0.65098039,0.01176471}\makebox(0,0)[lt]{\lineheight{1.25}\smash{\begin{tabular}[t]{l}$K_G(0)$\end{tabular}}}}%
    \put(0.17326789,0.02556635){\color[rgb]{0,0,0.99607843}\makebox(0,0)[lt]{\lineheight{1.25}\smash{\begin{tabular}[t]{l}$K_B(0)$\end{tabular}}}}%
    \put(0.30924556,0.16706724){\color[rgb]{0,0,0.97647059}\makebox(0,0)[lt]{\lineheight{1.25}\smash{\begin{tabular}[t]{l}$0$\end{tabular}}}}%
    \put(0.55835336,0.14785519){\color[rgb]{0.99215686,0,0.01176471}\makebox(0,0)[lt]{\lineheight{1.25}\smash{\begin{tabular}[t]{l}$0$\end{tabular}}}}%
    \put(0.6687077,0.18833262){\color[rgb]{0,0.6745098,0.01176471}\makebox(0,0)[lt]{\lineheight{1.25}\smash{\begin{tabular}[t]{l}$0$\end{tabular}}}}%
    \put(0.50319253,0.03211563){\color[rgb]{0,0,1}\makebox(0,0)[lt]{\lineheight{1.25}\smash{\begin{tabular}[t]{l}$\mp2$\end{tabular}}}}%
  \end{picture}%
\endgroup%

%% file: WHD.pdf_tex
\begingroup%
  \makeatletter%
  \providecommand\color[2][]{%
    \errmessage{(Inkscape) Color is used for the text in Inkscape, but the package 'color.sty' is not loaded}%
    \renewcommand\color[2][]{}%
  }%
  \providecommand\transparent[1]{%
    \errmessage{(Inkscape) Transparency is used (non-zero) for the text in Inkscape, but the package 'transparent.sty' is not loaded}%
    \renewcommand\transparent[1]{}%
  }%
  \providecommand\rotatebox[2]{#2}%
  \newcommand*\fsize{\dimexpr\f@size pt\relax}%
  \newcommand*\lineheight[1]{\fontsize{\fsize}{#1\fsize}\selectfont}%
  \ifx\svgwidth\undefined%
    \setlength{\unitlength}{406.52721186bp}%
    \ifx\svgscale\undefined%
      \relax%
    \else%
      \setlength{\unitlength}{\unitlength * \real{\svgscale}}%
    \fi%
  \else%
    \setlength{\unitlength}{\svgwidth}%
  \fi%
  \global\let\svgwidth\undefined%
  \global\let\svgscale\undefined%
  \makeatother%
  \begin{picture}(1,0.46506402)%
    \lineheight{1}%
    \setlength\tabcolsep{0pt}%
    \put(0,0){\includegraphics[width=\unitlength,page=1]{WHD.pdf}}%
    \put(0.33064737,0.35708799){\color[rgb]{0,0,0.01176471}\makebox(0,0)[lt]{\lineheight{1.25}\smash{\begin{tabular}[t]{l}$\cong$\end{tabular}}}}%
    \put(0.04326183,0.38780094){\color[rgb]{0,0,1}\makebox(0,0)[lt]{\lineheight{1.25}\smash{\begin{tabular}[t]{l}$0$\end{tabular}}}}%
    \put(0.03482116,0.11844188){\color[rgb]{0,0.68627451,0.01176471}\makebox(0,0)[lt]{\lineheight{1.25}\smash{\begin{tabular}[t]{l}$0$\end{tabular}}}}%
    \put(0.90026851,0.39814918){\color[rgb]{1,0,0.01176471}\makebox(0,0)[lt]{\lineheight{1.25}\smash{\begin{tabular}[t]{l}$0$\end{tabular}}}}%
    \put(0.4666977,0.44601328){\color[rgb]{0,0.73333333,0.01176471}\makebox(0,0)[lt]{\lineheight{1.25}\smash{\begin{tabular}[t]{l}$0$\end{tabular}}}}%
    \put(0.57120104,0.27858802){\color[rgb]{0,0,1}\makebox(0,0)[lt]{\lineheight{1.25}\smash{\begin{tabular}[t]{l}$\mp2$\end{tabular}}}}%
    \put(0.45125396,0.1723871){\color[rgb]{0,0,1}\makebox(0,0)[lt]{\lineheight{1.25}\smash{\begin{tabular}[t]{l}$\mp2$\end{tabular}}}}%
    \put(0.9200125,0.01643099){\color[rgb]{0,0.70588235,0.01176471}\makebox(0,0)[lt]{\lineheight{1.25}\smash{\begin{tabular}[t]{l}$0$\end{tabular}}}}%
    \put(0.59254156,0.00326828){\color[rgb]{0.98823529,0,0.01176471}\makebox(0,0)[lt]{\lineheight{1.25}\smash{\begin{tabular}[t]{l}$0$\end{tabular}}}}%
    \put(0.33503493,0.08725272){\color[rgb]{0,0,0.01176471}\makebox(0,0)[lt]{\lineheight{1.25}\smash{\begin{tabular}[t]{l}$\cong$\end{tabular}}}}%
    \put(0.7347768,0.2542356){\color[rgb]{0,0,0.01176471}\rotatebox{-90}{\makebox(0,0)[lt]{\lineheight{1.25}\smash{\begin{tabular}[t]{l}$\cong$\end{tabular}}}}}%
    \put(0,0){\includegraphics[width=\unitlength,page=2]{WHD.pdf}}%
  \end{picture}%
\endgroup%